\documentclass[final,onefignum,onetabnum]{siamart190516}

\usepackage{amsfonts}
\usepackage{amssymb}
\usepackage{bm}
\usepackage{enumitem}
\usepackage[normalem]{ulem}

\newsiamremark{remark}{Remark}
\newsiamremark{hypothesis}{Hypothesis}
\crefname{hypothesis}{Hypothesis}{Hypotheses}
\crefname{assumption}{Assumption}{Assumptions}
\crefalias{enumi}{assumption}
\newsiamthm{claim}{Claim}
\newsiamthm{problem}{Problem}

\headers{Robust Preconditioners for Multiple Saddle Point Problems}{A. Beigl, J. Sogn, and W. Zulehner}

\title{Robust Preconditioners for Multiple Saddle Point Problems and Applications to Optimal Control Problems\thanks{
\funding{This work was supported by the Austrian Science Fund (FWF) through the projects I3661-N27, P31048 and S11702-N23.}}}

\author{Alexander Beigl\thanks{Faculty of Mathematics, University of Vienna, 1090 Vienna, Austria 
  (\email{alexander.beigl@univie.ac.at}).}
\and Jarle Sogn\thanks{Johann Radon Institute for Computational Mathematics (RICAM), Johannes Kepler University, 4040 Linz, Austria 
  (\email{jarle.sogn@ricam.oeaw.ac.at}).}
\and Walter Zulehner\thanks{Institute of Computational Mathematics, Johannes Kepler University, 4040 Linz, Austria 
  (\email{zulehner@numa.uni-linz.ac.at}).}}

\usepackage{amsopn}

\newcommand{\norm}[1]{\left\|#1\right\|}
\newcommand{\inner}[2]{\left(#1,#2\right)}
\newcommand{\dual}[2]{\left<#1,#2\right>}
\newcommand{\wave}[1]{\partial_{tt}#1-\Delta #1}
\newcommand{\LLQ}{L^2\left(Q_T\right)}
\newcommand{\LLO}{L^2(\Omega)}

\ifpdf
\hypersetup{
  pdftitle={Robust Preconditioners for Multiple Saddle Point Problems and Applications to Optimal Control Problems},
  pdfauthor={A. Beigl, J. Sogn, and W. Zulehner}
}
\fi

\begin{document}

\maketitle

\begin{abstract}
In this paper we consider multiple saddle point problems with block tridiagonal Hessian in a Hilbert space setting. Well-posedness and the related issue of preconditioning are discussed.
We give a characterization of all block structured norms which ensure well-posedness of multiple saddle point problems as a helpful tool for constructing block diagonal preconditioners.
We subsequently apply our findings to a general class of PDE-constrained optimal control problems containing a regularization parameter $\alpha$ and derive $\alpha$-robust preconditioners for the corresponding optimality systems. Finally, we demonstrate the generality of our approach with two optimal control problems related to the heat and the wave equation, respectively. Preliminary numerical experiments support the feasibility of our method.
\end{abstract}



\section{Introduction}
In this paper we discuss the well-posedness of a particular class of saddle point problems in function
spaces and the related topic of robust preconditioning. We consider linear operator equations
\begin{equation}\label{eq:opeqA}
\mathcal{A}\bm{x} = \bm{b},
\end{equation}
where $\mathcal{A}:\bm{X}\longrightarrow \bm{X}'$ is a self-adjoint operator mapping from the product space $\bm{X} = X_1 \times X_2 \times \ldots \times X_n$ of Hilbert spaces $X_i$ into its dual space $\bm{X}'$.

In particular, we are interested in the case where 
$\mathcal{A} \colon \bm{X} \longrightarrow \bm{X}'$ is of $n$-by-$n$ block tridiagonal form
\begin{equation*}\label{eq:Atridi_intro}
    \mathcal{A} =     
    \begin{pmatrix} 
      A_1 & B_1'  &   & \\
      B_1 & -A_2  & \ddots &  \\
          &\ddots & \ddots & B_{n-1}' \\[1ex]
          &       & B_{n-1}   & (-1)^{n-1}A_n
  \end{pmatrix} ,
\end{equation*}
where $A_i \colon X_i \longrightarrow X_i'$, $B_i \colon X_i \longrightarrow X_{i+1}'$ are bounded linear operators, $B_i'$ is the adjoint of $B_i$, and, additionally, $A_i$ are self-adjoint and positive semi-definite. 
Under these assumptions, solutions to \cref{eq:opeqA} characterize multiple saddle points of the associated Lagrangian functional
\[
  \mathcal{L}(\bm{x}) = \frac{1}{2}\dual{\mathcal{A} \bm{x}}{\bm{x}} - \dual{\bm{b}}{\bm{x}},
\]
where $\dual{\cdot}{\cdot}$ denotes the duality product, see \cite{SogZul18} for more details of this interpretation.
The special case $n=2$ where \cref{eq:opeqA} is of the form
\begin{equation}\label{eq:classicalsaddlepointproblem}
\begin{pmatrix}
 A & B' \\ B & -C
\end{pmatrix}
\begin{pmatrix}
x \\ p
\end{pmatrix}
=
\begin{pmatrix}
f \\ g
\end{pmatrix}
\end{equation}
is usually referred to as a classical saddle point problem.

Saddle point problems in infinite-dimensional Hilbert spaces arise as the optimality systems of optimization problems in function spaces with a quadratic objective functional and constrained by a partial differential equation (PDE) or a system of PDEs. Other sources for such problems are mixed formulations of elliptic boundary value problems. For numerous applications of classical saddle point problems we refer to the seminal survey article \cite{BenGolLie05} and for applications of multiple saddle point problem we refer to \cite{SogZul18}.

Classical saddle point problems ($n = 2$) are well-studied, see \cite{BofBreFor13}. For $C = 0$, the well-known Brezzi conditions are sufficient and necessary conditions for well-posedness. This is generalized in \cite{Zul11}, where sufficient and necessary conditions, including the case $C \neq 0$,
are provided. The conditions in \cite{Zul11} also provided conditions for a robust preconditioner in the framework of operator preconditioning.

Multiple saddle point problems ($n > 2$) are less studied than classical saddle point problems. 
In \cite{SogZul18} a block diagonal preconditioner was introduced whose diagonal blocks 
consist of a sequence of so-called Schur complements. Well-posedness of \cref{eq:opeqA} could be shown with respect to the associated norm with robust estimates.
However, Schur complements do not always exist.
This becomes already apparent in the well-studied case \cref{eq:classicalsaddlepointproblem}, where $A$ needs to  be invertible only on the kernel of $B$. Then, of course, $A$ might be not invertible and, consequently, the classical Schur complement $S = C + B A^{-1} B'$ would not exist.
Therefore, a more general approach is undertaken here, where we consider general block diagonal preconditioners rather than the more restrictive class of preconditioners based on Schur complements.

An important field of applications are optimality systems of PDE-constrained optimization problems. In particular, optimal control problems are considered with objective functionals which contain a regularization term involving  some regularization parameter $\alpha$. Suitable Krylov subspace methods, e.g. the minimum residual method (MINRES), for solving the corresponding (discretized) optimality systems deteriorate for small $\alpha$ when using standard preconditioners. For most practical applications we have $0< \alpha \ll 1$, thus finding $\alpha$-robust preconditioners is essential.
For optimal control problems with an elliptic state equation $\alpha$-robust preconditioners are provided by \cite{SchZul07,PeaWat12,MarNieNor17,SogZul18}. Some time-depending problems are addressed in \cite{PeaStoWat12,LiuPea19}, however, the rigorous analysis of $\alpha$-robust preconditioners always required full observation (observation throughout the whole domain). A special case with a hyperbolic state equation was studied in \cite{BeiSchSogZul19}. There a robust preconditioner was obtained also for a problem with partial observation. Based on a new abstract theory we will derive $\alpha$-robust preconditioners without requiring full observation. The class of problems covered by the new approach include optimal control problems with elliptic, parabolic or hyperbolic state equations. The work presented here can be seen as an extension of ideas presented in  \cite{MarNieNor17}.

The paper is organized as follows. 
In \cref{sec:abstracttheory} the well-posedness of \cref{eq:opeqA} is addressed in general Hilbert spaces. The main result is contained in \cref{mainresult_new} which provides a characterization of robust block diagonal preconditioners for \cref{eq:opeqA}. 
This result can be seen as an extension of corresponding results in \cite{Zul11} to multiple saddle point problems.
In \cref{sec:optimalcontrol} the application of the abstract results to optimal control problems are discussed in general. 
\Cref{sec:examples} contains particular examples of optimal control problems with parabolic resp.~hyperbolic state equations. 
Preliminary numerical results are reported in \cref{sec:numerics}. 
Finally, a few auxiliary results needed for the abstract analysis are provided in \cref{sec:app}.

\section{Abstract theory}
\label{sec:abstracttheory}

We introduce some notation which will be used throughout the paper.

\emph{Notation} 1: For a real Hilbert space $X$ with inner product $\inner{\cdot}{\cdot}_X$, the duality pairing in its dual space $X'$ will be denoted by $\dual{\cdot}{\cdot}_X$ where we omit the subscript when the space is clear from the context.

For a bounded linear operator $B:X\longrightarrow Y'$, where $X$ and $Y$ are Hilbert spaces, its adjoint $B':Y\longrightarrow X'$ is given by
\[
 \dual{B'y}{x} = \dual{Bx}{y}\quad\text{for all} \ x\in X,\, y\in Y.
\]
A bounded linear operator $A:X\longrightarrow X'$ is said to be self-adjoint, respectively positive semi-definite, if
\[
  \dual{Ay}{x} = \dual{Ax}{y},
  \quad \text{resp.} \quad
  \dual{Ax}{x} \geq 0,
  \quad \text{for all} \  x, y  \in X.
\]
The operator $A$ is positive definite (coercive) if
\[
\dual{Ax}{x}\geq \sigma\norm{x}^2_X\quad\text{for all} \ x\in X
\]
for some positive constant $\sigma$.

Let $\bm{X} = X_1 \times X_2 \times \ldots \times X_n$ be the product space of Hilbert spaces $X_i$ for $i=1,2,\ldots,n$, endowed with the canonical inner product
\[
\inner{\bm x}{\bm y}_{\bm X} = \inner{x_1}{y_1}_{X_1}+\inner{x_2}{y_2}_{X_2}+\ldots + \inner{x_n}{y_n}_{X_n},
\]
and let the linear operator $\mathcal{A} \colon \bm{X} \longrightarrow \bm{X}'$ be of $n$-by-$n$ block tridiagonal form
\begin{equation*}\label{eq:Atridi}
    \mathcal{A} =     
    \begin{pmatrix} 
      A_1 & B_1'  &   & \\
      B_1 & -A_2  & \ddots &  \\
          &\ddots & \ddots & B_{n-1}' \\[1ex]
          &       & B_{n-1}   & (-1)^{n-1}A_n
  \end{pmatrix} ,
\end{equation*}
where $A_i \colon X_i \longrightarrow X_i'$, $B_i \colon X_i \longrightarrow X_{i+1}'$ are bounded linear operators, and, additionally, $A_i$ are self-adjoint and positive semi-definite. Here, as usual, we identify the dual space $\bm{X}'$ with $X_1' \times X_2' \times \ldots \times X_n'$.

For a given right-hand side $\bm b \in \bm X'$, we consider the linear operator equation
\begin{equation} \label{Mxb_new}
  \mathcal{A} \bm{x} = \bm{b}.
\end{equation}

We introduce two linear operators associated to $\mathcal{A}$:
\begin{equation*}
    \mathcal{D} =     
    \begin{pmatrix} 
      A_1 &   &   & \\
       & A_2  &  &  \\
          & & \ddots \\
          &       &  & A_n
  \end{pmatrix} 
  \quad \text{and} \quad
    \mathcal{B} =     
    \begin{pmatrix} 
      0 & B_1'  &   & \\
      B_1 & 0  & \ddots &  \\
          &\ddots & \ddots & B_{n-1}' \\
          &       & B_{n-1}   & 0
  \end{pmatrix} .
\end{equation*}
Observe that $\mathcal D$ and $\mathcal B$ are self-adjoint, and, additionally, $\mathcal D$ is positive semi-definite. Furthermore, let
\begin{equation*}
    \widetilde{\bm x} =     
    \begin{pmatrix} 
      x_1 \\
      - x_2\\
      \vdots \\
      (-1)^{n-1} \, x_n
  \end{pmatrix} 
  \  \text{for} \ 
    \bm x =     
    \begin{pmatrix} 
      x_1 \\
      x_2\\
      \vdots \\
      x_n
  \end{pmatrix} 
  \quad \text{and} \quad
    \widetilde{\mathcal D} =     
    \begin{pmatrix} 
      A_1 &   &   & \\
       & - A_2  &  &  \\
          & & \ddots \\
          &       &  & (-1)^{n-1} \, A_n
  \end{pmatrix} .
\end{equation*}
These notations are used in the following analysis. 

We start with the analysis of the uniqueness of a solution to \cref{Mxb_new}.
\begin{lemma}\label{le:uniqueness}
$\ker \mathcal{A} = \ker \mathcal{D} \cap \ker \mathcal{B}$.
\end{lemma}
\begin{proof}
With the notations introduced above we have for all $\bm x \in \bm X$:
\[
  \langle \mathcal A \bm x, \widetilde{\bm x} \rangle
    = \langle (\widetilde{\mathcal D} + \mathcal B) \bm x, \widetilde{\bm x} \rangle
    =  \langle \mathcal D \bm x, \bm x \rangle 
       + \langle \mathcal B \bm x, \widetilde{\bm x} \rangle
    =  \langle \mathcal D \bm x, \bm x \rangle .
\]
Therefore, if $\bm x \in \ker \mathcal{A}$, then
\[
  \langle \mathcal{D} \bm{x},\bm{x} \rangle = \langle \mathcal{A} \bm{x},\widetilde{\bm{x}}\rangle = 0,
\]
which implies that $x \in \ker \mathcal{D}$, since $\mathcal{D}$ is self-adjoint and positive semi-definite. Furthermore, since $\ker \widetilde{\mathcal D} = \ker \mathcal D$, it follows that $\mathcal B \bm x = \mathcal A \bm x - \widetilde{\mathcal D} \bm x = 0$. This concludes the proof of $\ker \mathcal{A} \subset \ker \mathcal{D} \cap \ker \mathcal{B}$. 

On the other hand, if $\bm x \in \ker \mathcal{D} \cap \ker \mathcal{B}$, then $\bm x \in \ker \widetilde{\mathcal D}$ and, consequently, $\mathcal A \bm x = \widetilde{\mathcal D} \bm x + \mathcal B \bm x = 0$, which shows that $\ker \mathcal{D} \cap \ker \mathcal{B} \subset \ker \mathcal{A}$.
\end{proof}

The next theorem deals with the well-posedness of \cref{Mxb_new}.
\begin{theorem} \label{mainresult_new}
If there are positive constants $\underline{c}$ and $\overline{c}$ such that
\begin{equation} \label{Mc_new}
   \underline{c} \, \|\bm x\|_{\bm X} \le \|\mathcal{A} \bm x\|_{\bm X'} \le \overline{c} \, \|\bm x\|_{\bm X}
   \quad \text{for all} \ \bm x \in \bm X,
\end{equation}
then
\begin{equation} \label{Mgamma_new}
   \underline{\gamma} \, \|\bm x\|_{\bm X}^2 \le \langle \mathcal{D} \bm x,\bm x\rangle + \|\mathcal{B} \bm x\|_{\bm X'}^2 \le \overline{\gamma} \, \|\bm x\|_{\bm X}^2
   \quad \text{for all} \ \bm x \in \bm X,
\end{equation}
with positive constants $\underline{\gamma}$ and $\overline{\gamma}$ which depend only on $\underline{c}$ and $\overline{c}$.
Vice versa, if there are positive constants $\underline{\gamma}$ and $\overline{\gamma}$ such that \cref{Mgamma_new} holds, then \cref{Mc_new} holds with positive constants $\underline{c}$ and $\overline{c}$ which depend only on $\underline{\gamma}$ and $\overline{\gamma}$.
\end{theorem}

\begin{proof}
First we show that \cref{Mc_new} implies \cref{Mgamma_new}. 
For the estimate from above in \cref{Mgamma_new} observe that
\[
  \langle \mathcal{D} \bm x, \bm x \rangle 
  = \langle \mathcal A \bm x, \widetilde{\bm x} \rangle 
  \le \|\mathcal A \bm x\|_{\bm X'} \, \|\widetilde{\bm x}\|_{\bm X} 
  \le \overline{c} \, \|\bm x\|_{\bm X} \, \|\widetilde{\bm x}\|_{\bm X} =  \overline{c} \, \|\bm x\|_{\bm X}^2.
\]
In order to estimate $\|\mathcal B \bm x\|_{\bm X'}$ we use 
\[
  \|\mathcal B \bm x\|_{\bm X'}
  = \|(\mathcal A - \widetilde{\mathcal D}) \bm x\|_{\bm X'} 
  \le \|\mathcal A \bm x\|_{\bm X'} + \|\widetilde{\mathcal D} \bm x\|_{\bm X'} 
   =  \|\mathcal A \bm x\|_{\bm X'} + \|\mathcal D \bm x\|_{\bm X'} .
\]
Since
\[
  \langle \mathcal{D} \bm x,\bm y \rangle^2
    \le \langle \mathcal{D} \bm x, \bm x\rangle \, \langle \mathcal{D} \bm y, \bm y\rangle
    \le \overline{c}^2 \, \|\bm x\|_{\bm X}^2 \, \|\bm y\|_{\bm X}^2,
\]
it follows that
\[
  \|\mathcal D \bm x\|_{\bm X'} 
     =  \sup_{\bm 0 \neq \bm y \in \bm X} \frac{\langle \mathcal{D} \bm x,\bm y \rangle}{\|\bm y\|_{\bm X}}
  \le \overline{c} \, \|\bm x\|_{\bm X},
\]
which allows to complete the estimate of  $\|\mathcal B \bm x\|_{\bm X'}$:
\[
  \|\mathcal B \bm x\|_{\bm X'}
  \le \|\mathcal A \bm x\|_{\bm X'} + \|\mathcal D \bm x\|_{\bm X'}
  \le 2 \, \overline{c} \, \|\bm x\|_{\bm X}.
\]
The estimates of $\langle \mathcal{D} \bm x, \bm x \rangle$ and $\|\mathcal B \bm x\|_{\bm X'}$ lead directly to the estimate from above in \cref{Mgamma_new} with $\overline{\gamma} = \overline{c} + 4 \, \overline{c}^2$.

For showing the estimate from below in \cref{Mgamma_new} we start with the following argument:
\[
  \langle \mathcal{D} \bm x,\bm y \rangle^2
    \le \langle \mathcal{D} \bm x, \bm x\rangle \, \langle \mathcal{D} \bm y, \bm y\rangle
    \le \langle \mathcal{D} \bm x, \bm x\rangle \, \| \mathcal{D} \bm y\|_{\bm X'} \| \bm y\|_{\bm X} 
    \le \overline{c} \, \langle \mathcal{D} \bm x, \bm x\rangle \, \|\bm y\|_{\bm X}^2,
\]
which implies
\[
  \|\mathcal{D} \bm x\|^2 
     =  \sup_{\bm 0 \neq \bm y \in \bm X} \frac{\langle \mathcal{D} \bm x,\bm y \rangle^2}{\|\bm y\|_{\bm X}^2}
    \le \overline{c} \, \langle \mathcal{D} \bm x, \bm x\rangle.
\]
Therefore,
\begin{align*}
  \underline{c} \, \|\bm x\|_{\bm X} 
    & \le \|\mathcal A \bm x\|_{\bm X'} 
       =  \|(\widetilde{\mathcal D} + \mathcal B) \bm x\|_{\bm X'}
      \le \|\widetilde{\mathcal D} \bm x\|_{\bm X'} + \|\mathcal B \bm x\|_{\bm X'}
       =  \|\mathcal D \bm x\|_{\bm X'} + \|\mathcal B \bm x\|_{\bm X'} \\
    & \le \overline{c}^{1/2} \, \langle \mathcal{D} \bm x, \bm x\rangle^{1/2} +  \|\mathcal B \bm x\|_{\bm X'}
      \le (\overline{c} + 1)^{1/2} \, \left(\langle \mathcal{D} \bm x, \bm x\rangle + \|\mathcal B \bm x\|_{\bm X'}^2 \right)^{1/2},
\end{align*}
from which the estimate from below in \cref{Mgamma_new} follows for $\underline{\gamma} = \underline{c}^2/(\overline{c}+1)$.

It remains to show that \cref{Mgamma_new} implies \cref{Mc_new}. For the estimate from above in \cref{Mc_new} we again use the triangle inequality and obtain
\[
  \|\mathcal A \bm x\|_{\bm X'} 
      \le \|\mathcal D \bm x\|_{\bm X'} + \|\mathcal B \bm x\|_{\bm X'} ,
\]
see above. Since
\[
  \langle \mathcal{D} \bm x,\bm y \rangle^2
    \le \langle \mathcal{D} \bm x, \bm x\rangle \, \langle \mathcal{D} \bm y, \bm y\rangle
    \le \overline{\gamma} \, \langle \mathcal{D} \bm x, \bm x\rangle \, \|\bm y\|_{\bm X}^2,
\]
it follows that
\begin{equation} \label{ANormRay_new}
  \|\mathcal D \bm x\|_{\bm X'}^2  
     =  \sup_{\bm 0 \neq \bm y \in \bm X} \frac{\langle \mathcal{D} \bm x,\bm y \rangle^2}{\|\bm y\|_{\bm X}^2}
    \le \overline{\gamma} \, \langle \mathcal{D} \bm x, \bm x\rangle,
\end{equation}
which allows to complete the estimate of  $\|\mathcal A \bm x\|_{\bm X'}$:
\begin{align*}
  \|\mathcal A \bm x\|_{\bm X'} 
      & \le \overline{\gamma}^{1/2} \, \langle \mathcal{D} \bm x, \bm x\rangle^{1/2} + \|\mathcal B \bm x\|_{\bm X'} 
      \le (\overline{\gamma} + 1)^{1/2} \, \left( \langle \mathcal{D} \bm x, \bm x\rangle + \|\mathcal B \bm x\|_{\bm X'}^2\right)^{1/2} \\
      & \le (\overline{\gamma} + 1)^{1/2}  \, \overline{\gamma}^{1/2} \, \|\bm x\|_{\bm X} .
\end{align*}
Then the estimate from above in \cref{Mc_new} follows for $\overline{c} = \big[\overline{\gamma}\, (\overline{\gamma} + 1)\big]^{1/2}$.

For the estimate from below in \cref{Mc_new}, we start with the following two estimates:
\begin{equation} \label{Below1_new}
  \|\mathcal A \bm x\|_{\bm X'} 
    \ge \big| \|\mathcal B \bm x\|_{\bm X'} - \|\mathcal D \bm x\|_{\bm X'} \big|
\end{equation}
and
\begin{equation} \label{Below2_new}
  \|\mathcal{A} \bm x\|_{\bm X'} 
   \ge \frac{\langle \mathcal{A} \bm x,\widetilde{\bm x}\rangle}{\|\widetilde{\bm x}\|_{\bm X}}
    =  \frac{\langle \mathcal{D} \bm x,\bm x \rangle}{\|\bm x\|_{\bm X}}
   \ge \frac{1}{\overline{\gamma}} \, \frac{\| \mathcal D \bm x\|_{\bm X'}^2}{\|\bm x\|_{\bm X}}
   \quad \text{for} \ \bm x \neq \bm 0.
\end{equation}
The first estimate follows from the triangle inequality. For the second estimate we used \cref{ANormRay_new}. 
Next we need to estimate $\|\bm x\|_{\bm X}$ from above in terms of $\|\mathcal D \bm x\|_{\bm X'}$ and $\|\mathcal B \bm x\|_{\bm X'}$:
From \cref{Mgamma_new} and
\[
  \langle \mathcal D \bm x,\bm x\rangle \le \frac{1}{2 \, \varepsilon} \, \|\mathcal D \bm x\|_{\bm X'}^2 + \frac{\varepsilon}{2} \, \|\bm x\|_{\bm X}^2
\]
it follows that
\[
  \underline{\gamma} \, \|\bm x\|_{\bm X}^2 
    \le \langle \mathcal D \bm x,\bm x\rangle + \|\mathcal B \bm x\|_{\bm X'}^2 
    \le \frac{1}{2 \, \varepsilon} \, \|\mathcal D \bm x\|_{\bm X'}^2 + \frac{\varepsilon}{2} \, \|\bm x\|_{\bm X}^2
    + \|\mathcal B \bm x\|_{\bm X'}^2.
\]
For $\varepsilon = \underline{\gamma}$ we obtain
\[
  \frac{\underline{\gamma}}{2} \, \| \bm x\|_{\bm X}^2 
    \le \frac{1}{2 \underline{\gamma}} \, \|\mathcal D \bm x\|_{\bm X'}^2 + \|\mathcal B \bm x\|_{\bm X'}^2
    \le \max\left(\frac{1}{2 \underline{\gamma}},1\right) \, \left(\|\mathcal D \bm x\|_{\bm X'}^2 + \|\mathcal B \bm x\|_{\bm X'}^2\right),
\]
which implies
\[
  \delta \, \|\bm x\|_{\bm X}^2 
    \le \|\mathcal D \bm x\|_{\bm X'}^2 + \|\mathcal B \bm x\|_{\bm X'}^2
  \quad \text{with} \quad
  \delta =  \min \left(\underline{\gamma}^2, \underline{\gamma}/2 \right) .
\]
With this estimate we obtain from the estimates \cref{Below1_new} and \cref{Below2_new} for $\bm x \neq 0$:
\[
  \|\mathcal{A} \bm x\|_{\bm X'} 
   \ge \big| \|\mathcal B \bm x\|_{\bm X'} - \|\mathcal D \bm x\|_{\bm X'} \big|
  = |\eta - \xi| \, (\|\mathcal D \bm x\|_{\bm X'}^2 + \|\mathcal B \bm x\|_{\bm X'}^2)^{1/2}
  \ge \delta^{1/2} \,  |\eta - \xi| \, \|\bm x\|_{\bm X} 
\]
and
\[
  \|\mathcal{A} \bm x\|_{\bm X'} 
   \ge \frac{1}{\overline{\gamma}} \, \frac{\| \mathcal D \bm x\|_{\bm X'}^2}{\|\bm x\|_{\bm X}}
  = \frac{1}{\overline{\gamma}} \, \xi^2 \, \frac{\|\mathcal D \bm x\|_{\bm X'}^2 + \|\mathcal B \bm x\|_{\bm X'}^2}{\|\bm x\|_{\bm X}}
  \ge (\delta/\overline{\gamma}) \, \xi^2 \, \|\bm x\|_{\bm X}
\]
with
\[
  \xi = \frac{\| \mathcal D \bm x\|_{\bm X'}}{(\|\mathcal D \bm x\|_{\bm X'}^2 + \|\mathcal B \bm x\|_{\bm X'}^2)^{1/2}}
  \quad \text{and} \quad
  \eta  = \frac{\|\mathcal B \bm x\|_{\bm X'}}{(\|\mathcal D \bm x\|_{\bm X'}^2 + \|\mathcal B \bm x\|_{\bm X'}^2)^{1/2}} .
\]
Note that $\xi$ and $\eta$ are well-defined for $x\notin \ker\mathcal{A}$ by \cref{le:uniqueness}.

By combining these two estimates we obtain
\begin{align*}
  \|\mathcal{A} \bm x\|_{\bm X'} 
    & \ge (\delta/\overline{\gamma}) 
          \, \max\left(|\eta - \xi|, \xi^2\right)  \, \|\bm x\|_{\bm X} ,
\end{align*}
where we used that
\[
  \delta/\overline{\gamma}
    = \begin{cases}
        \underline{\gamma}^2/\overline{\gamma} \le \underline{\gamma} = \sqrt{\delta} & \quad \text{if} \ \underline{\gamma} \le 1/2, \\
        \underline{\gamma}/(2\overline{\gamma})  \le 1/2 \le \sqrt{\underline{\gamma}/2}  = \sqrt{\delta}& \quad \text{if} \ \underline{\gamma} \ge 1/2. \\
      \end{cases}
\]
Observe that $\xi^2 + \eta^2 = 1$ and
\[
  \varphi(\xi,\eta) \ge \min \{\varphi(x,y) \colon x,y \ge 0,\ x^2 + y^2 =1 \} \ge 0.29
\]
with
$
  \varphi(x,y) = \max \left(|y - x|, x^2\right).
$
Therefore
\[
  \|\mathcal{A} \bm x\|_{\bm X'}   \ge \underline{c} \, \|\bm x\|_{\bm X}
  \quad \text{with} \quad
  \underline{c} = 0.29 \, (\delta/\overline{\gamma})
    = 0.29 \, \min \left(\underline{\gamma}^2, \underline{\gamma}/2 \right) / \overline{\gamma}. 
\]
which concludes the proof.
\end{proof}

Assume that we have self-adjoint and positive definite bounded linear operators $P_i \colon X_i \longrightarrow X_i'$ inducing inner products on $X_i$ via
\[
 \inner{x_i}{y_i}_{P_i}=\dual{P_i x_i}{y_i} \quad\text{for all} \  x_i, y_i  \in X_i.
\]
Then the block diagonal operator $\mathcal P \colon \bm X \longrightarrow \bm X'$, given by
\begin{equation*}
    \mathcal{P} =     
    \begin{pmatrix} 
      P_1 &      &        &  \\
          & P_2  &        &  \\
          &      & \ddots &  \\[1ex]
          &      &        & P_n
  \end{pmatrix} ,
\end{equation*}
defines an inner product on $\bm X$, called the $\mathcal{P}$-inner product, by virtue of
\[
\inner{\bm x}{\bm y}_\mathcal{P} = \dual{\mathcal{P} \bm x}{\bm y}\quad\text{for all} \  \bm x, \bm y  \in \bm X.
\]
The associated equivalent norm on $\bm X$, called the $\mathcal{P}$-norm, will be denoted by $\norm{\cdot}_\mathcal{P}$.

With this notation, we want to express \cref{Mgamma_new} in a more convenient form. For the Hilbert space $\bm X$ equipped with the $\mathcal{P}$-norm it follows by \cref{le:appendix1} that
\[
\norm{\mathcal{B}\bm x}^2_{\bm X'}
=\sup_{0\neq \bm y\in \bm X}\frac{\dual{\mathcal{B}\bm x}{\bm y}^2}{\dual{\mathcal{P}\bm y}{\bm y}} 
= \dual{\mathcal{B}\mathcal{P}^{-1}\mathcal{B}\bm x}{\bm x}\quad\text{for all}\ \bm x\in \bm X.
\]
Therefore, the condition \cref{Mgamma_new} of \cref{mainresult_new} can be written in the short form
\begin{equation}\label{eq:spectralequivalence}
  \mathcal{P} \sim \mathcal{D} + \mathcal{B} \mathcal{P}^{-1}  \mathcal{B},
\end{equation}
using the following notation:

\emph{Notation} 2: Let $M,N:X\longrightarrow X'$ be two self-adjoint bounded linear operators. Then the following hold:
\begin{enumerate}
 \item $M\leq N$ if and only if
 \[
  \dual{Mx}{x}\leq \dual{Nx}{x}\quad\text{for all} \ x\in X.
 \]
 \item $M\lesssim N$ if and only if there is a constant $c\geq 0$ such that $M\leq cN$.
 \item $M\sim N$ if and only if $M\lesssim N$ and $N\lesssim M$. In this case we call $M$ and $N$ spectrally equivalent.
\end{enumerate}
If the operators $M$ and $N$ depend on some parameters (like a regularization parameter $\alpha$ or a discretization parameter $h$), then we additionally assume that the involved constants are independent of those parameters.

With this notation, \cref{mainresult_new} offers a result on robust preconditioning of \cref{Mxb_new}: For the Hilbert space $\bm X$ equipped with the $\mathcal{P}$-norm, given by a block diagonal operator $\mathcal{P}:\bm X\longrightarrow \bm X'$ satisfying the relation \cref{eq:spectralequivalence}, there exist parameter-independent constants $\underline{c}$, $\overline{c}$ such that \cref{Mc_new} holds. Since
\[
\norm{\mathcal{A} \bm x}^2_{\bm X'}  =
\dual{\mathcal{A}\mathcal{P}^{-1}\mathcal{A}\bm x}{\bm x} =
\dual{\mathcal{A}\bm x}{\mathcal{P}^{-1}\mathcal{A}\bm x} = 
\norm{\mathcal{P}^{-1}\mathcal{A}\bm x}^2_{\mathcal{P}},
\]
well-posedness \cref{Mc_new} can be written as
\begin{equation*}
  \underline{c} \, \|\bm x\|_{\mathcal{P}}
   \le \|\mathcal{P}^{-1}\mathcal{A} \bm x\|_{\mathcal{P}}
   \le \overline{c} \, \|\bm x\|_{\mathcal{P}}
   \quad \text{for all} \ \bm x  \in \bm X .
\end{equation*}
Consequently, it follows for the condition number
\[
 \kappa(\mathcal{P}^{-1}\mathcal{A}) = \|\mathcal{P}^{-1}\mathcal{A}\|_{L(\bm X,\bm X)} \|(\mathcal{P}^{-1}\mathcal{A})^{-1}\|_{L(\bm X,\bm X)}  \leq \frac{\overline{c}}{\underline{c}}. 
\]
Therefore, the task of finding a good preconditioner $\mathcal{P}:\bm X\longrightarrow \bm X'$ for the system \cref{Mxb_new} translates to choosing inner products $\inner{\cdot}{\cdot}_{P_i}$ on the Hilbert spaces $X_i$ such that the condition \cref{eq:spectralequivalence} is satisfied.

We now illustrate \cref{eq:spectralequivalence} for the three interesting cases $n\in\{2,3,4\}$.
\subsection{The case $n=2$}
Let
\[
  \mathcal{A} = \begin{pmatrix} A_1 & B_1' \\ B_1 & -A_2 \end{pmatrix}, \quad
  \mathcal{D} = \begin{pmatrix} A_1 & 0 \\ 0 & A_2 \end{pmatrix}, \quad
  \mathcal{B} = \begin{pmatrix} 0 & B_1' \\ B_1 & 0 \end{pmatrix}, \quad
  \mathcal{P} = \begin{pmatrix} P_1 & 0 \\ 0 & P_2 \end{pmatrix}.
\]
Then
\[
   \mathcal{B} \mathcal{P}^{-1} \mathcal{B}
   = \begin{pmatrix} 
       B_1' P_2^{-1} B_1 & 0 \\
       0 & B_1 P_1^{-1} B_1' 
     \end{pmatrix}
\]
and the spectral relation
\[
  \mathcal{P} \sim \mathcal{D} + \mathcal{B} \mathcal{P}^{-1} \mathcal{B}
\]
is equivalent to
\[
  P_1 \sim A_1 + B_1' P_2^{-1} B_1 
  \quad  \text{and} \quad
  P_2 \sim A_2 + B_1 P_1^{-1} B_1'.
\]
Thus, we recover the result from \cite{Zul11}.
\subsection{The case $n=3$}
Let
\[
  \mathcal{A} = \begin{pmatrix} 
                  A_1 & B_1' & 0 \\ 
                  B_1 & -A_2 & B_2' \\
                  0 & B_2 & A_3
                \end{pmatrix},
\]
\[
  \mathcal{D} = \begin{pmatrix} 
                  A_1 & 0 & 0 \\ 
                  0 & A_2 & 0 \\
                  0 & 0 & A_3
                \end{pmatrix}, \quad
  \mathcal{B} = \begin{pmatrix} 
                  0 & B_1' & 0 \\ 
                  B_1 & 0 & B_2' \\
                  0 & B_2 & 0
                \end{pmatrix}, \quad
  \mathcal{P} = \begin{pmatrix} 
                  P_1 & 0 & 0 \\ 
                  0 & P_2 & 0 \\
                  0 & 0 & P_3
                \end{pmatrix}.
\]
Then
\[
  \mathcal{B} \mathcal{P}^{-1} \mathcal{B} 
    = \begin{pmatrix} 
        B_1' P_2^{-1} B_1 & 0 & B_1' P_2^{-1} B_2' \\
        0 & B_1 P_1^{-1} B_1' + B_2' P_3^{-1} B_2 & 0\\
        B_2 P_2^{-1} B_1 & 0 & B_2 P_2^{-1} B_2'
     \end{pmatrix}
\]
and the spectral relation
\[
  \mathcal{P} \sim \mathcal{D} + \mathcal{B} \mathcal{P}^{-1} \mathcal{B}
\]
is equivalent to
\begin{equation*}
  \begin{pmatrix} 
    P_1 &  0 \\ 
    0 & P_3
  \end{pmatrix}
  \sim 
  \begin{pmatrix} 
     A_1 + B_1' P_2^{-1} B_1 & B_1' P_2^{-1} B_2' \\
     B_2 P_2^{-1} B_1 & A_3 + B_2 P_2^{-1} B_2' 
  \end{pmatrix}
\end{equation*}
and
\begin{equation*}
 P_2 \sim A_2 + B_1 P_1^{-1} B_1' + B_2' P_3^{-1} B_2.
\end{equation*}
\subsection{The case $n=4$}
Let
\[
  \mathcal{A} = \begin{pmatrix} 
                  A_1 & B_1' & 0 & 0 \\ 
                  B_1 & -A_2 & B_2' & 0 \\
                  0 & B_2 & A_3 & B_3' \\
                  0 & 0 & B_3 & -A_4
                \end{pmatrix},
\]
\[
  \mathcal{D} = \begin{pmatrix} 
                  A_1 & 0 & 0 & 0 \\ 
                  0 & A_2 & 0 & 0 \\
                  0 & 0 & A_3 & 0 \\
                  0 & 0 & 0 & A_4
                \end{pmatrix}, \,
  \mathcal{B} = \begin{pmatrix} 
                  0 & B_1' & 0 & 0 \\ 
                  B_1 & 0 & B_2' & 0 \\
                  0 & B_2 & 0 & B_3' \\
                  0 & 0 & B_3 & 0
                \end{pmatrix}, \,
  \mathcal{P} = \begin{pmatrix} 
                  P_1 & 0 & 0 & 0 \\ 
                  0 & P_2 & 0 & 0 \\
                  0 & 0 & P_3 & 0 \\
                  0 & 0 & 0 & P_4
                \end{pmatrix}.
\]
Then
\begin{multline*}
  \mathcal{B} \mathcal{P}^{-1} \mathcal{B} \\
   = \begin{pmatrix} 
        B_1' P_2^{-1} B_1 & 0 & B_1' P_2^{-1} B_2' & 0 \\
        0 & B_1 P_1^{-1} B_1' + B_2' P_3^{-1} B_2 & 0 & B_2' P_3^{-1} B_3'\\
        B_2 P_2^{-1} B_1 & 0 & B_2 P_2^{-1} B_2' + B_3' P_4^{-1} B_3 & 0 \\
        0 & B_3 P_3^{-1} B_2 & 0 & B_3 P_3^{-1} B_3'
     \end{pmatrix}
\end{multline*}
and the spectral relation
\[
  \mathcal{P} \sim \mathcal{D} + \mathcal{B} \mathcal{P}^{-1} \mathcal{B}
\]
is equivalent to
\begin{equation}\label{eq:cond1}
  \begin{pmatrix} 
    P_1 &  0 \\ 
    0 & P_3
  \end{pmatrix}
  \sim 
  \begin{pmatrix} 
     A_1 + B_1' P_2^{-1} B_1 & B_1' P_2^{-1} B_2' \\
     B_2 P_2^{-1} B_1 & A_3 + B_2 P_2^{-1} B_2' + B_3' P_4^{-1} B_3
   \end{pmatrix}
\end{equation}
and
\begin{equation}\label{eq:cond2}
  \begin{pmatrix} 
    P_2 &  0 \\ 
    0 & P_4
  \end{pmatrix}
  \sim 
  \begin{pmatrix} 
     A_2 + B_1 P_1^{-1} B_1' + B_2' P_3^{-1} B_2 & B_2' P_3^{-1} B_3' \\
     B_3 P_3^{-1} B_2 & A_4 + B_3 P_3^{-1} B_3'
   \end{pmatrix}.
\end{equation}

\section{Application to optimal control problems}\label{sec:optimalcontrol}
We are now going to apply our theory to general abstract optimal control problems constrained by linear partial differential equations:

For given data $d$ and fixed $\alpha>0$, we consider the minimization problem of finding a state $y$ and control $u$ which minimize the functional 
 \begin{equation}\label{eq:minimizationfunctional}
  J:Y\times U\longrightarrow \mathbb{R},\quad J(y,u) = \frac{1}{2}\norm{Ty-d}^2_O+\frac{\alpha}{2}\norm{u}^2_U,
 \end{equation}
subject to the constraint
\begin{equation}\label{eq:stateequation}
 Ky + Cu = g.
\end{equation}
Here, $Y$ denotes the state space, $U$ is the control space, and $O$ is the observation space.
The bounded linear observation operator $T:Y\longrightarrow O$ in \cref{eq:minimizationfunctional} maps the state to the measurements.

The state equation \cref{eq:stateequation} is given in terms of the bounded linear operators
\[
 K: Y\longrightarrow M'\,\text{(state operator)}\quad\text{and}\quad C:U\longrightarrow M'\,\text{(control operator)}.
\]
Here, we assume that the test space $M$ is a product space of Hilbert spaces where the first space is the same function space as used for the control,
\[
M=U\times R.
\]
The components of $K$ will be denoted by the bounded linear operators
\[
K_U:Y\longrightarrow U',\quad K_R:Y\longrightarrow R',
\]
such that $Ky=(K_Uy,K_Ry)^\top$. Typically, the components of the state operator represent the differential expression and side conditions, such as boundary, and (or) initial conditions, respectively. For illustrative examples we refer to \cref{sec:examples}.

The crucial assumption of the control operator $C$ is that it acts only on the first line of the state equation, that is, our considered state equation \cref{eq:stateequation} has the particular form
\begin{equation}\label{eq:consideredstateequation}
\begin{pmatrix}K_U \\ K_R\end{pmatrix}y + \begin{pmatrix} \mathcal{I}_U \\ 0 \end{pmatrix} u = \begin{pmatrix}g_U\\g_R\end{pmatrix}.
\end{equation}
Here, we used the following notation:

\emph{Notation} 3: The inner product in a Hilbert space $X$ induces a self-adjoint and positive definite bounded linear operator $\mathcal{I}_X \colon X \longrightarrow X'$, given by
\[
  \dual{\mathcal{I}_X x}{y} = \inner{x}{y}_X
  \quad \text{for all} \ x, y  \in X,
\]
whose inverse is usually called the Riesz isomorphism associated to the Hilbert space $X$.

\begin{remark}
We stress that the following treatment does not exclude the trivial case $R=\{0\}$ which corresponds to full control distributed on $M$.
\end{remark}

The optimality system for the constrained optimization problem \cref{eq:minimizationfunctional,eq:consideredstateequation} reads as follows:

Find $(y,u,p_U,p_R)\in Y\times U\times U\times R$ such that
 \begin{equation}\label{eq:optoriginal}
\begin{pmatrix}
   T'\mathcal{I}_OT & 0 & K_U' & K_R' \\
   0 & \alpha \mathcal{I}_U & \mathcal{I}_U & 0 \\
   K_U & \mathcal{I}_U & 0 & 0 \\
   K_R & 0 & 0 & 0
 \end{pmatrix}
  \begin{pmatrix}
   y \\ u \\ p_U \\ p_R
  \end{pmatrix}
= \begin{pmatrix}
   \inner{d}{T\cdot}_O \\ 0 \\ g_U \\ g_R
  \end{pmatrix}
.
 \end{equation}
\begin{remark}
Note that $T'\mathcal{I}_OT:Y\longrightarrow Y'$ is given by $y\mapsto \inner{Ty}{T\cdot}_O$.
\end{remark}
After a reordering, the optimality can equivalently be written in tridiagonal form:

Find $(u,p_U,y,p_R)\in U\times U\times Y\times R$ such that
\begin{equation}\label{eq:optreordered}
 \begin{pmatrix}
  \alpha \mathcal{I}_U & \mathcal{I}_U & 0 & 0 \\
  \mathcal{I}_U & 0 & K_U & 0 \\
  0 & K_U' & T'\mathcal{I}_OT & K_R' \\
  0 & 0 & K_R & 0
 \end{pmatrix}
 \begin{pmatrix}
  u \\ p_U \\ y \\ p_R
 \end{pmatrix}
=
 \begin{pmatrix}
  0 \\ g_U \\ \inner{d}{T\cdot}_O \\ g_R
  \end{pmatrix}
.
\end{equation}
It is obvious that the spectral relations \cref{eq:cond1} and \cref{eq:cond2} strongly depend on the properties of the involved operators $K_U$, $K_R$ (and $T$). We are going to make the following assumptions:
\begin{enumerate}[label=(K\arabic*)]
\item \label{K1} The operator $K:Y\longrightarrow M'$, defined by $y\mapsto (K_Uy,K_Ry)^\top$, has closed range and is injective, or, equivalently, there exists a positive constant $c_K$ such that
\begin{equation}\label{eq:normequivalence2_new}
\norm{y}_Y  \leq c_K\norm{Ky}_{M'}=c_K\sqrt{\norm{K_Uy}^2_{U'} + \norm{K_Ry}^2_{R'}}
\quad \text{for all} \ y\in Y.
\end{equation}
\item \label{K2} The operator $K_R:Y\longrightarrow R'$ is surjective, or, equivalently, there exists a positive constant $c_R$ such that
\[
    \sup_{0 \neq y \in Y} \frac{\dual{K_R y}{r}}{\norm{y}_Y} \ge c_R \, \norm{r}_R
    \quad \text{for all} \ r\in R .
\]
\end{enumerate}
The \cref{K2} will also be considered in the stronger form:
\begin{enumerate}[label=(K\arabic*')]\addtocounter{enumi}{1}
\item \label{K2'} The operator $K_R|_{\ker K_U}:\ker K_U\longrightarrow R'$ is surjective, or, equivalently, there exists a positive constant $c_R$ such that
\[
    \sup_{0 \neq y \in \ker K_U} \frac{\dual{K_R y}{r}}{\norm{y}_Y} \ge c_R \, \norm{r}_R
    \quad\text{for all}\ r\in R .
\]
\end{enumerate}
\begin{remark}
Since we assumed that $K_U$, $K_R$ are bounded linear operators, it follows from \cref{eq:normequivalence2_new} that $\norm{K\cdot}_{M'}$ induces an equivalent norm on $Y$.
\end{remark}
\begin{remark}
\Cref{eq:normequivalence2_new}  can be seen as a natural a-priori estimate for a linear partial differential equation of the form $Ky=g$ , which states that \emph{if} a unique solution to $Ky=g$ exists, then it needs to be bounded by the data $g\in M'$. 
\end{remark}

The next theorem deals with the well-posedness of \cref{eq:optreordered} and offers a corresponding robust preconditioner. The derivation of the preconditioner is constructive in the following sense: Having a good guess for three out of four inner products on the Hilbert spaces $X_i\in\{U,U,Y,R\}$ leading to a robust preconditioner for the optimality system \cref{eq:optreordered}, the remaining fourth inner product follows almost as a gift from the spectral relation \cref{eq:spectralequivalence}.

To be more precise, for the Hilbert spaces $X_1=U$, $X_2=U$, $X_4=R$, we choose inner products corresponding to the operators
\[
P_1=\alpha \mathcal{I}_U,\quad P_2 = \alpha^{-1}\mathcal{I}_U,\quad P_4 = \mathcal{I}_R,
\]
respectively. With this choice the condition \cref{eq:cond1} reads
\[
  \begin{pmatrix} 
    \alpha \, \mathcal{I}_U &  0 \\ 
    0 & P_3
  \end{pmatrix}
  \sim 
  \begin{pmatrix} 
     2\alpha \, \mathcal{I}_U & \alpha \, K_U \\
     \alpha\, K_U' & T'\mathcal{I}_OT + \alpha \, K_U' \mathcal{I}_U^{-1} K_U + K_R' \mathcal{I}_R^{-1} K_R
   \end{pmatrix}.
\]
Then, by \cref{le:appendix3}, the only possible candidate for $P_3$ is given by (up to spectral equivalence)
\begin{equation}\label{eq:necessarycondition}
P_3 =  T'\mathcal{I}_OT + \alpha \, K_U' \mathcal{I}_U^{-1} K_U + K_R' \mathcal{I}_R^{-1} K_R.
\end{equation}
The proof of the next theorem guarantees that \cref{eq:necessarycondition} is not only necessary but also sufficient.
\begin{theorem}
  \label{theo:mainAppPrec}
 Let $\alpha>0$ and assume that \cref{K1,K2} are satisfied.
 
 The linear operator $\mathcal{A}:U\times U\times Y\times R = \bm{X}\longrightarrow \bm{X}'$ defined in \cref{eq:optreordered} is a self-adjoint isomorphism. Furthermore, for the Hilbert space $\bm{X}$ endowed with the inner product
 \begin{equation*}
  \inner{\bm{x}}{\bm{y}}_{\mathcal{P}} = \dual{\mathcal{P}\bm{x}}{\bm{y}}  \quad \text{for all} \ \bm{x},\bm{y} \in \bm{X},
 \end{equation*}
where $\mathcal{P}:\bm{X}\longrightarrow \bm{X}'$ is given by
\begin{equation*}
    \mathcal{P} =     
    \begin{pmatrix} 
      \alpha \mathcal{I}_U &      &        &  \\
          & \alpha^{-1}\mathcal{I}_U  &        &  \\
          &      & T'\mathcal{I}_OT + \alpha \, K_U' \mathcal{I}_U^{-1} K_U + K_R' \mathcal{I}_R^{-1} K_R &  \\
          &      &        & \mathcal{I}_R
  \end{pmatrix},
\end{equation*}
there exist positive constants $\underline{c}$ and $\overline{c}$, both independent of $\alpha\in(0,1]$, such that 
\begin{equation}\label{eq:isoineq}
  \underline{c} \, \|\bm x\|_{\mathcal{P}}
   \le \|\mathcal{P}^{-1}\mathcal{A} \bm x\|_{\mathcal{P}}
   \le \overline{c} \, \|\bm x\|_{\mathcal{P}}
   \quad \text{for all} \ \bm x  \in \bm X .
\end{equation}

Under the stronger \cref{K2'}, the constants in \cref{eq:isoineq} are independent of all $\alpha>0$.
\end{theorem}
\begin{proof}
 Denoting
 \[
  P_3 = T'\mathcal{I}_OT + \alpha \, K_U' \mathcal{I}_U^{-1} K_U + K_R' \mathcal{I}_R^{-1} K_R,
 \]
by \cref{mainresult_new,eq:cond1,eq:cond2}, it suffices to show
 \[
  \begin{pmatrix} 
    \alpha \, \mathcal{I}_U &  0 \\ 
    0 & P_3
  \end{pmatrix}
  \sim 
  \begin{pmatrix} 
     2\alpha \, \mathcal{I}_U & \alpha \, K_U \\
     \alpha\, K_U' & T'\mathcal{I}_OT + \alpha \, K_U' \mathcal{I}_U^{-1} K_U + K_R' \mathcal{I}_R^{-1} K_R
   \end{pmatrix}
\]
and
\begin{equation*}\label{eq:productnormeq}
  \begin{pmatrix} 
    \alpha^{-1} \, \mathcal{I}_U &  0 \\ 
    0 & \mathcal{I}_R
  \end{pmatrix}
  \sim 
  \begin{pmatrix} 
     \alpha^{-1} \, \mathcal{I}_U + K_U P_3^{-1} K_U' & K_U P_3^{-1} K_R' \\
     K_R P_3^{-1} K_U' & K_R P_3^{-1} K_R'
   \end{pmatrix}.
\end{equation*}
The first condition is satisfied, since we have for the associated (dual) Schur complement
\begin{multline*}
  T'\mathcal{I}_OT + \alpha \, K_U' \mathcal{I}_U^{-1} K_U + K_R' \mathcal{I}_R^{-1} K_R - \frac{1}{2} \alpha \, K_U' \mathcal{I}_U^{-1} K_U \\
   =  T'\mathcal{I}_OT + \frac{\alpha}{2} \, K_U' \mathcal{I}_U^{-1} K_U + K_R' \mathcal{I}_R^{-1} K_R 
\sim P_3 .
\end{multline*}

Concerning the second condition, by \cref{le:appendix2}, we have
\[
  K_U P_3^{-1} K_U'
    = K_U [T'\mathcal{I}_OT + \alpha \, K_U' \mathcal{I}_U^{-1} K_U + K_R' \mathcal{I}_R^{-1} K_R]^{-1} K_U' \le \alpha^{-1} \, \mathcal{I}_U,
\]
which implies
\begin{equation}\label{eq:equivalence}
  \alpha^{-1} \, \mathcal{I}_U \sim \alpha^{-1} \, \mathcal{I}_U + K_U P_3^{-1} K_U'.
\end{equation}
We also have
\[
  K_R P_3^{-1} K_R'
    = K_R [T'\mathcal{I}_OT + \alpha \, K_U' \mathcal{I}_U^{-1} K_U + K_R' \mathcal{I}_R^{-1} K_R]^{-1} K_R' \le \mathcal{I}_R.
\]
So, in order to ensure
\begin{equation} \label{nonsing_new}
  \mathcal{I}_R \sim K_R P_3^{-1} K_R',
\end{equation}
it suffices to show
$
  \mathcal{I}_R \lesssim K_R P_3^{-1} K_R',
$
or, equivalently, 
\[
  \sup_{0\neq y \in Y} \frac{\langle K_R y,r\rangle}{\langle P_3 y,y\rangle^{1/2}} \gtrsim \norm{r}_R\quad\text{for all}\ r\in R.
\]
This easily follows from \cref{K2}, since
\[
  \dual{P_3 y}{y} \lesssim \norm{y}^2_Y\quad \text{for all} \ y\in Y,
\]
under the mild condition that $\alpha$ is uniformly bounded, e.g., $\alpha \le 1$ and, therefore,
\[
  \sup_{0\neq y \in Y} \frac{\dual{K_R y}{r}}{\dual{P_3 y}{y}^{1/2}} 
  \gtrsim
  \sup_{0 \neq y \in Y} \frac{\dual{K_R y}{r}}{\norm{y}_Y} \gtrsim \norm{r}_R \quad\text{for all}\ r\in R.
\]
Therefore, \cref{nonsing_new} holds and $K_R P_3^{-1} K_R'$ is non-singular.
Then, by \cref{le:appendix2}, it follows that
\[
 K_R' [K_R P_3^{-1} K_R']^{-1} K_R \leq P_3,
\]
and as a consequence,
\[
 K_U P_3^{-1} K_R' [K_R P_3^{-1} K_R']^{-1}K_R P_3^{-1} K_U' \leq K_U P_3^{-1} K_U'.
\]
Therefore, we obtain for the associated (primal) Schur complement:
\[
  \alpha^{-1} \, \mathcal{I}_U \leq
  \alpha^{-1} \, \mathcal{I}_U + K_U P_3^{-1} K_U'
  - K_U P_3^{-1} K_R' [K_R P_3^{-1} K_R']^{-1}K_R P_3^{-1} K_U'.
\]
The assertion then follows from \cref{eq:equivalence} and \cref{le:appendix3}.

Under no restrictions on $\alpha$ we have
\[
  \dual{P_3 y}{y} \lesssim \norm{y}^2_Y 
  \quad \text{for all} \ y \in \ker K_U,
\]
and, therefore,
\[
  \sup_{0\neq y \in Y} \frac{\dual{K_R y}{r}}{\langle P_3 y,y\rangle^{1/2}}
  \gtrsim
  \sup_{0 \neq y \in \ker K_U} \frac{\dual{K_R y}{r}}{\norm{y}_Y} \gtrsim \norm{r}_R \quad\text{for all}\ r\in R
\]
under the stronger assumption \cref{K2'}.
\end{proof}
\begin{remark}
  \Cref{theo:mainAppPrec} also holds true under the relaxed condition that $T:Y\longrightarrow O$ is invertible on the kernel of $K:Y\longrightarrow M'$, if $\ker K\neq \{0\}$, as it was done in \cite{MarNieNor17}.
\end{remark}

\subsection{Brezzi constants}
In the original ordering the optimality system \cref{eq:optoriginal} can be phrased as a classical saddle point problem:

Find $x=(y,u)\in Y\times U = X$ and $p=(p_U,p_R)\in U\times R = M$ such that
\begin{equation}\label{eq:classicalsystem}
\begin{pmatrix}
A & B' \\
B & 0
\end{pmatrix}
\begin{pmatrix}
x \\ p
\end{pmatrix}
=
\begin{pmatrix}
f \\ g
\end{pmatrix},
\end{equation}
where $A:X\longrightarrow X'$, $B:X\longrightarrow M'$ are given by
\begin{equation*}\label{eq:matricesAB}
A = \begin{pmatrix} T'\mathcal{I}_OT & 0 \\ 0 & \alpha \mathcal{I}_U\end{pmatrix},
\quad
B = \begin{pmatrix} K_U & \mathcal{I}_U \\ K_R & 0 \end{pmatrix},
\end{equation*}
and $f\in X'$, $g\in M'$ are given by
\begin{equation*}
f(w) = \inner{d}{T z }_O,\quad g(q) = g_U(q_U) + g_R(q_R),
\end{equation*}
for all $w=(z, v )\in X$, $q=(q_U,q_R)\in M$, respectively.

In this setting, by Brezzi's theorem (see \cite{BofBreFor13}), well-posedness of \cref{eq:classicalsystem} is equivalent to the following (Brezzi) conditions:
\begin{enumerate}
 \item \label{Brezzi1} The linear operator $A$ is bounded: There exists a positive constant $c_A$ such that
 \[
  \dual{A x}{w}\leq c_A\norm{x}_X\norm{w}_X \quad\text{for all}\ x,w\in X.
 \]
 \item \label{Brezzi2}The linear operator $B$ is bounded: There exists a positive constant $c_B$ such that
 \[
  \dual{B w}{q}\leq c_B\norm{w}_X\norm{q}_M \quad\text{for all}\ w\in X,\ q\in M.
 \]
 \item \label{Brezzi3}The linear operator $A$ is coercive on $\ker B=\{w\in X:\dual{Bw}{q}=0\text{ for all }q\in M\}$:
 There exists a positive constant $\gamma_0$ such that
 \[
  \dual{A w}{w}\geq \gamma_0\norm{w}^2_X\quad\text{for all} \ w\in \ker B.
 \]
 \item \label{Brezzi4}The linear operator $B$ satisfies the inf-sup condition: There exists a positive constant $k_0$ such that
 \[
  \sup_{0\neq w\in X}\frac{\dual{Bw}{q}}{\norm{w}_X}\geq k_0 \norm{q}_M\quad\text{for all} \ q\in M.
 \]
\end{enumerate}
The constants appearing in the four conditions are referred to as Brezzi constants. \Cref{theo:mainAppPrec} already guarantees the existence of $\alpha$-independent Brezzi constants. Their particular values are provided by the following theorem.
\begin{theorem}[Brezzi constants]
 Let $\alpha>0$ and assume that \cref{K1,K2'} are satisfied. For the Hilbert spaces $X$, $M$ endowed with the norms
\begin{equation*}
 \begin{aligned}
\norm{w}^2_{X,\alpha} &= \norm{T z }^2_O + \alpha \norm{K_U  z }^2_{U'} + \norm{K_R  z }^2_{R'} +\alpha\norm{ v }^2_U,\quad &&w=( z , v ) \in X,\\
\norm{q}^2_{M,\alpha} &= \alpha^{-1}\norm{q_U}^2_U +\norm{q_R}^2_R, &&q=(q_U,q_R)\in M,
\end{aligned}
\end{equation*}
respectively, the Brezzi conditions are satisfied with 
 \begin{equation*}
  c_A = 1,\quad c_B=\sqrt{2},\quad \gamma_0=\frac{1}{2},\quad k_0 = \frac{1}{\sqrt{\norm{T}^2_{L(Y,O)}c_K^2+1}},
 \end{equation*}
where the positive constant $c_K$ is from \cref{eq:normequivalence2_new}.
\end{theorem}
\begin{proof}
 To prove the first condition, we estimate with the Cauchy-Schwarz inequality as follows: Let $x=(y,u), w=( z , v )\in X$, then
 \begin{align*}
  \dual{A x}{w} &= \inner{Ty}{T z }_O+\alpha\inner{u}{ v }_U
  \leq\sqrt{\norm{Ty}^2_O+\alpha\norm{u}^2_U}\sqrt{\norm{T z }^2_O+\alpha\norm{ v }^2_U}\\
  &\leq\norm{x}_{X,\alpha}\norm{w}_{X,\alpha}.
 \end{align*}

 In an analogous manner we show the second condition: Let $w=( z , v )\in X$, $q=(q_U,q_R)\in M$, then
 \begin{align*}
  \dual{Bw}{q} &= \dual{K_U z }{q_U} + \inner{ v }{q_U}_U + \dual{K_R z }{q_R}\\
  &\leq \sqrt{\alpha}\left(\norm{K_U z }_{U'}+\norm{ v }_U\right)\frac{1}{\sqrt{\alpha}}\norm{q_U}_U+\norm{K_R z }_{R'}\norm{q_R}_R\\
 &\leq\sqrt{\alpha\left(\norm{K_U z }_{U'}+\norm{ v }_U\right)^2 + \norm{K_R z }^2_{R'}}\norm{q}_{M,\alpha}
  \leq \sqrt{2}\norm{w}_{X,\alpha}\norm{q}_{M,\alpha}.
 \end{align*}

 In order to prove the coercivity estimate in the third condition we first note that
 \begin{align*}
  \ker B &= \{w\in X: \dual{Bw}{q}=0 \text{ for all }q\in M\} \\
  &= \{( z , v )\in Y\times U : (K_U z ,K_R z )^\top=(-\mathcal{I}_U v ,0)^\top\}.
 \end{align*}
 Therefore, for $w=( z , v )\in\ker B$ we obtain
 \begin{align*}
  \dual{A w}{w} &= \norm{T z }^2_O + \alpha\norm{ v }^2_U 
  = \norm{T z }^2_O + \frac{\alpha}{2}\norm{K_U z }^2_{U'} +  \frac{\alpha}{2}\norm{ v }^2_U \\
  &= \frac{1}{2}\norm{T z }^2_O+\frac{1}{2}\norm{w}^2_{X,\alpha}
   \geq\frac{1}{2}\norm{w}^2_{X,\alpha}.
 \end{align*}

 To prove the fourth condition let $0\neq q = (q_U,q_R)\in M$ be arbitrary. Under the assumption that $K_R|_{\ker K_U}$ is surjective, we can choose $\hat{y}\in \ker K_U$ such that $K_R\hat{y}=\mathcal{I}_Rq_R\in R'$. Then for $\hat{w}=(\hat{y},\alpha^{-1}q_U)$  we obtain
 \begin{align*}
  \dual{B\hat{w}}{q}&=\dual{K_U\hat{y}}{q_U} + \alpha^{-1}\norm{q_U}^2_U + \dual{K_R\hat{y}}{q_R}\\
  &= \alpha^{-1}\norm{q_U}^2_U + \norm{q_R}^2_R
  = \norm{q}^2_{M,\alpha}.
 \end{align*}
 The assertion then follows from the estimate
 \begin{align*}
  \norm{\hat{w}}^2_{X,\alpha} &= \norm{T\hat{y}}^2_O + \alpha \norm{K_U\hat{y}}^2_{U'} + \norm{K_R\hat{y}}^2_{R'} + \alpha\norm{\alpha^{-1}q_U}^2_U\\
  &\leq \norm{T}^2_{L(Y,O)}\norm{\hat{y}}^2_Y + \norm{q_R}^2_R + \alpha^{-1}\norm{q_U}^2_U\\
  &\leq \norm{T}^2_{L(Y,O)}c_K^2\left(\norm{K_U\hat{y}}^2_{U'} + \norm{K_R\hat{y}}^2_{R'}\right)+ \norm{q_R}^2_R + \alpha^{-1}\norm{q_U}^2_U\\
  &\leq \left(\norm{T}^2_{L(Y,O)}c_K^2+1\right)\left(\norm{q_R}^2_R + \alpha^{-1}\norm{q_U}^2_U\right)\\
  &= \left(\norm{T}^2_{L(Y,O)}c_K^2+1\right)\norm{q}^2_{M,\alpha},
 \end{align*}
where we used \cref{eq:normequivalence2_new} in the second inequality.
\end{proof}

\section{Examples}\label{sec:examples}
Under the \cref{K1,K2}, respectively \labelcref{K2'}, \cref{theo:mainAppPrec} guarantees well-posedness of the optimality system \cref{eq:optreordered} and proposes a robust preconditioner. The question arises for which particular applications, that is, linear PDEs, these conditions are fulfilled.

Usually, the operators $K_U$ and $K_R$ represent the differential operator and the side conditions of the PDE, respectively. Elliptic control problems of this form (with $R=\{0\}$) have been considered in \cite{MarNieNor17}, where the space for the control and the test space for the non-standard variational formulation in strong form of the state equation coincide. In this sense the setting in \cite{MarNieNor17} fits into our framework and, therefore, we will focus here as an alternative on time-dependent problems in the following. 

In the two examples to come we have a linear and bijective state operator $K=(K_U,K_R)^\top:Y\longrightarrow M'$, where $M$ is a product space of Hilbert spaces $U$, $R$ and $Y$ in the first place is just a linear space.
There is a natural way of introducing a canonical Hilbert space structure on $Y$ such that the \cref{K1,K2,K2'} are satisfied.
\begin{lemma}\label{le:onetoone}
 Under the assumptions for $Y$, $M$, and $K:Y\longrightarrow M'$ from above, $Y$ is a Hilbert space endowed with the inner product
\begin{align*}\label{eq:innerproduct}
 \inner{y}{ z }_Y &= \dual{Ky}{\mathcal{I}^{-1}_MK z } \\
 &= \dual{K_Uy}{\mathcal{I}^{-1}_U K_U z } + \dual{K_Ry}{\mathcal{I}^{-1}_R K_R z } \quad\text{for all}\ y, z \in Y,
\end{align*}
where $\mathcal{I}_M$, $\mathcal{I}_U$ and $\mathcal{I}_R$ represent the canonical inner products on $M$, $U$ and $R$, respectively, see Notation 3.

Moreover, for $Y$ equipped with the inner product $\inner{\cdot}{\cdot}_Y$, the state operator $K$ satisfies \cref{K1,K2,K2'}. 
\end{lemma}
\begin{proof}
 Observe that $\dual{\cdot}{\mathcal{I}^{-1}_M\cdot}$ defines an inner product on $M'$ by the definition of $\mathcal{I}_M$. Since $K$ is linear and injective, it follows that $\inner{\cdot}{\cdot}_Y=\dual{K\cdot}{\mathcal{I}^{-1}_M K\cdot}$ defines an inner product on $Y$.
 
 In order to show that $Y$ is complete with respect to $\norm{\cdot}_Y = \sqrt{\inner{\cdot}{\cdot}_Y}=\norm{K\cdot}_{M'}$ let $(y_k)_{k \in \mathbb{N}}$ be a Cauchy sequence in $Y$. Then $(Ky_k)_{k \in \mathbb{N}}$ is a Cauchy sequence in $M'$ which possesses a limit in $M'$ denoted by $g$. Since $K$ is bijective, there exists a unique $y\in Y$ such that $Ky = g$. Consequently,
 \[
  \norm{y_k-y}_Y = \norm{K(y_k-y)}_{M'} = \norm{Ky_k-g}_{M'}\rightarrow 0\quad \text{as}\quad k\rightarrow \infty.
 \]
 Thus the Cauchy sequence is converging and thus the space $Y$ is complete.
 
 Note that \cref{K1} is trivially satisfied by the definition of $\norm{\cdot}_Y$, that is,
 \[
  \norm{y}_Y = \norm{Ky}_{M'} = \sqrt{\norm{K_Uy}^2_{U'} + \norm{K_Ry}^2_{R'}}\quad\text{for all}\ y\in Y.
 \]
 \Cref{K2'} is fulfilled as well since for any $g_R\in R'$ the system $Ky=(0,g_R)^\top$ is uniquely solvable. Finally, \labelcref{K2'} implies \labelcref{K2}.
\end{proof}

This lemma is needed for the following two examples of the heat equation and the wave equation considered over a bounded domain $\Omega\subset\mathbb{R}^N$, $N\geq1$, with Lipschitz boundary $\partial \Omega$, and over a finite time interval $(0,T)$. Throughout the remaining of the paper we will introduce the space-time cylinder by $Q_T=\Omega\times(0,T)$ and its lateral surface by $\Gamma_T=\partial\Omega\times[0,T]$.

\subsection{Heat equation} 
Consider the heat equation with homogeneous Dirichlet boundary conditions on $\Gamma_T$,
\begin{equation*}
\begin{aligned}
 \partial_t y - \Delta y &= f\quad&&\text{in}&&Q_T,\\
 y &= 0 &&\text{on}&&\Gamma_T,\\
 y(0) &= y_0 &&\text{in}&&\Omega,
 \end{aligned}
\end{equation*}
for given data $f$, $y_0$. For this problem we introduce the following function spaces. 
$L^2(D)$ denotes the standard Lebesgue space of square-integrable functions on a domain $D$. 
$H_0^1(\Omega)$ is the subspace of the standard Sobolev space $H^1(\Omega)$ of functions on $\Omega$ with vanishing trace on $\partial \Omega$. We use the inner product $\inner{v}{w}_{H^1_0(\Omega)}=\inner{\nabla v}{\nabla w}_{L^2(\Omega)}$ for functions $v,w\in H^1_0(\Omega)$. Moreover, let 
\[
  H(\Delta,\Omega) = \{ v \in L^2(\Omega) : \Delta v \in L^2(\Omega) \}
\]
with inner product $(v,w)_{H(\Delta,\Omega)} = (v,w)_{L^2(\Omega)} + (\Delta v,\Delta w)_{L^2(\Omega)}$. 
Finally, for a Hilbert space $H$, $L^2((0,T);H)$ denotes the Bochner space of square-integrable functions from $(0,T)$ to $H$ with inner product
\[
  (f,g)_{L^2((0,T);H)} = \int_0^T (f(t),g(t))_H \ dt.  
\]
Then the well-known solution theory for this initial-boundary value problem can be summarized in the spirit of \cref{le:onetoone}.

\begin{theorem}[{\cite[Chapter 3, Theorem 2.1]{Lad85}}]\label{th:onetooneheat}
 Let $\Omega\subset \mathbb{R}^N$ be a bounded domain with Lipschitz boundary $\partial \Omega$ and $T>0$ finite. Define
\begin{align}
 Y & = \{ y \in L^2((0,T);H_0^1(\Omega) \cap H(\Delta,\Omega)) : \partial_t u \in L^2(Q_T) \}, \label{eq:defYheat} \\
 M & = L^2(Q_T) \times H^1_0(\Omega) . \nonumber
 \end{align}
 Then the operator 
 \begin{equation*}\label{eq:Kheat}
 K: Y\longrightarrow M',\quad y\mapsto
 \begin{pmatrix}
  \inner{\partial_t y - \Delta y}{\cdot}_{L^2(Q_T)} \\
  \inner{y(0)}{\cdot}_{H^1_0(\Omega)}
 \end{pmatrix}
\end{equation*}
is linear and bijective.
\end{theorem}

Combining \cref{theo:mainAppPrec} and \cref{le:onetoone} we obtain the following result:
\begin{corollary}\label{co:heat}
 Let $Y$, $M$, and $K:Y\longrightarrow M'$ be as in \cref{th:onetooneheat}. Then $Y$ is a complete space with respect to the norm $\|\cdot\|_Y$, given by
 \[
  \norm{y}^2_Y = \norm{\partial_ty-\Delta y}^2_{L^2(Q_T)} + \norm{y(0)}^2_{H^1_0(\Omega)}\quad \text{for all}\ y \in Y.
 \] 
 Moreover, for any decomposition of $M$ as a product space of Hilbert spaces $U$ and $R$ and any bounded linear observation operator $T:Y\longrightarrow O$, there exists a unique minimizer of the constrained optimization problem \cref{eq:minimizationfunctional,eq:consideredstateequation} which is characterized by the solution of the optimality system \cref{eq:optoriginal}.
\end{corollary}

\subsection*{Optimal control problem for the heat equation}
 Let $\omega$ be a non-empty open subset of $\Omega$ and denote $q_T = \omega \times (0,T) \subset Q_T$. We consider an optimal control problem of minimizing a tracking-type quadratic cost functional with (possibly) limited observation plus a regularization term, where the constraint is the heat equation. 
 More precisely, for $Y$ given by \cref{eq:defYheat}, we want to minimize the functional
\begin{equation*}
J: Y \times L^2(Q_T) \longrightarrow \mathbb{R},\qquad J(y,u) = \frac{1}{2} \norm{y-d}^2_{L^2(q_T)} + \frac{\alpha}{2}\norm{u}^2_{L^2(Q_T)}, 
\end{equation*}
subject to
\begin{equation*}
 \begin{aligned}
  \inner{\partial_t y - \Delta y}{q_U}_{L^2(Q_T)} + \inner{u}{q_U}_{L^2(Q_T)} &= 0 
  && \forall\ q_U\in L^2(Q_T)
  ,\\
  \inner{y(0)}{q_R}_{H^1_0(\Omega)} &= \inner{y_0}{q_R}_{H^1_0(\Omega)} &&\forall\ q_R\in H^1_0(\Omega),\\
 \end{aligned}
\end{equation*}
for given initial value $y_0\in H^1_0(\Omega)$ and data $d\in L^2(q_T)$.

The optimality system then reads as follows:

Find $(y,u,p_U,p_R)\in Y\times U\times U\times R $ such that
\begin{equation}\label{eq:optsystemheat}
 \begin{pmatrix}
  T'\mathcal{I}_O T & 0 & K_U' & K_R'  \\
  0 & \alpha \mathcal{I}_U & \mathcal{I}_U & 0 \\
  K_U & \mathcal{I}_U & 0 & 0 \\
  K_R & 0 & 0 & 0 
 \end{pmatrix}
 \begin{pmatrix}
  y\\u\\p_U\\p_R
 \end{pmatrix}
 =
\begin{pmatrix}
 \inner{d}{T\cdot}_{L^2(q_T)} \\ 0 \\ 0 \\ \inner{y_0}{\cdot}_{H^1_0(\Omega)}
\end{pmatrix}
,
\end{equation}
with the spaces $Y$ given by \cref{eq:defYheat},
\[
 U = L^2(Q_T),\quad R = H^1_0(\Omega), \quad O = L^2(q_T) ,
\]
and the operators
\begin{equation}
 \begin{aligned}
  T &: Y\longrightarrow O,\quad&&y\mapsto y|_{q_T},\\
  K_U&:Y\longrightarrow U',&&y\mapsto \inner{\partial_t y - \Delta y}{\cdot}_{L^2(Q_T)} ,\\
  K_R&:Y\longrightarrow R',&& y\mapsto \inner{y(0)}{\cdot}_{H^1_0(\Omega)}.
 \end{aligned}
\end{equation}

It follows from \cref{co:heat} that the system \cref{eq:optsystemheat} is well-posed. Additionally, by \cref{theo:mainAppPrec}, the $\mathcal{P}$-norm leading to an $\alpha$-robust preconditioner is given by
\begin{multline*}
 \norm{( z , v ,q_U,q_R)}^2_\mathcal{P} 
 = \norm{ z }^2_{L^2(q_T)} + \alpha \norm{\partial_t  z  - \Delta  z }^2_{L^2(Q_T)}+\norm{\nabla  z (0)}^2_{\LLO}\\
+ \alpha \norm{ v }^2_{L^2(Q_T)} + \frac{1}{\alpha}\norm{q_U}^2_{L^2(Q_T)} + \norm{\nabla q_R}^2_{\LLO}.
\end{multline*}

\subsection{Wave equation} \label{ssec:wave}
Consider the wave equation with homogeneous Dirichlet boundary conditions on $\Gamma_T$,
\begin{equation*}
\begin{aligned}
 \partial_{tt} y - \Delta y &= f\quad&&\text{in}&&Q_T,\\
 y &= 0 &&\text{on}&&\Gamma_T,\\
 y(0) &= y_0 &&\text{in}&&\Omega,\\
 \partial_ty(0) &= y_1 &&\text{in}&&\Omega,
 \end{aligned}
\end{equation*}
for given data $f$, $y_0$, $y_1$. For this problem we need one further function space. $C([0,T]; H)$ denotes the space of continuous functions from $[0,T]$ to a Hilbert space $H$. 
Then the well-known solution theory for this initial-boundary value problem can be summarized in the spirit of \cref{le:onetoone}.
\begin{theorem}[{\cite[Chapter 3, Theorem 8.2]{LioMag72a}}]\label{th:onetoonewave}
 Let $\Omega\subset \mathbb{R}^N$ be a bounded domain with Lipschitz boundary $\partial \Omega$ and $T>0$ finite. Define
\begin{align}
 Y & = \{ y\in C([0,T];H^1_0(\Omega)):\partial_ty\in C([0,T];L^2(\Omega)),\,\partial_{tt}y-\Delta y \in L^2(Q_T)\}, \label{eq:defYwave}\\
 M & = L^2(Q_T)\times H^1_0(\Omega) \times L^2(\Omega).\nonumber
\end{align}
 Then the operator
 \begin{equation}\label{eq:Kwave}
 K: Y\longrightarrow M',\quad y\mapsto
 \begin{pmatrix}
  \inner{\partial_{tt}y-\Delta y}{\cdot}_{L^2(Q_T)} \\
  \inner{y(0)}{\cdot}_{H^1_0(\Omega)}\\
  \inner{\partial_ty(0)}{\cdot}_{L^2(\Omega)}
 \end{pmatrix}
 ,
\end{equation}
is linear and bijective.
\end{theorem}
Combining \cref{theo:mainAppPrec} and \cref{le:onetoone} we obtain the following result:
\begin{corollary}\label{co:wave}
 Let $Y$, $M$, and $K:Y\longrightarrow M'$ be as in \cref{th:onetoonewave}. Then $Y$ is a complete space with respect to
 \[
  \norm{y}^2_Y = \norm{\partial_{tt}y-\Delta y}^2_{L^2(Q_T)} + \norm{y(0)}^2_{H^1_0(\Omega)}+ \norm{\partial_ty(0)}^2_{L^2(\Omega)}\quad \text{for all}\ y \in Y.
 \] 
 Moreover, for any decomposition of $M$ as a product space of Hilbert spaces $U$ and $R$ and any bounded linear observation operator $T:Y\longrightarrow O$ there exists a unique minimizer of the constrained optimization problem \cref{eq:minimizationfunctional,eq:consideredstateequation} which is characterized by the solution of the optimality system \cref{eq:optoriginal}.
\end{corollary}
In a previous work we considered the problem of controlling the initial condition of the wave equation,
\[
 y(0) = u,\quad u\in U = H^1_0(\Omega),
\]
see \cite{BeiSchSogZul19} for further details. This particular study was the starting point for our current work. Here, we want to demonstrate the flexibility of our approach and consider in the following a constrained optimization problem for the wave equation where we control the differential expression
\[
 \partial_{tt} y - \Delta y + u = 0,\quad u\in U = L^2(Q_T).
\]

\subsection*{Optimal control problem for the wave equation}
 Let $\omega$ be a non-empty open subset of $\Omega$ and denote $q_T = \omega \times (0,T) \subset Q_T$. We consider an optimal control problem of minimizing a tracking-type quadratic cost functional with (possibly) limited observation plus a regularization term where the constraint is the wave equation. More precisely, for $Y$ given by \cref{eq:defYwave}, we want to minimize the functional
\begin{equation*}
J: Y \times L^2(Q_T)\longrightarrow \mathbb{R},\qquad J(y,u) = \frac{1}{2} \norm{y-d}^2_{L^2(q_T)} + \frac{\alpha}{2}\norm{u}^2_{L^2(Q_T)}, 
\end{equation*}
subject to
\begin{equation*}
 \begin{aligned}
  \inner{\partial_{tt} y - \Delta y }{q_U}_{L^2(Q_T)} + \inner{u}{q_U}_{L^2(Q_T)} &= 0 &&\forall\ q_U\in L^2(Q_T),\\
  \inner{y(0)}{q_{R_1}}_{H^1_0(\Omega)} &= \inner{y_0}{q_{R_1}}_{H^1_0(\Omega)} &&\forall\ q_{R_1}\in H^1_0(\Omega),\\
  \inner{\partial_t y(0)}{q_{R_2}}_{L^2(\Omega)} &= \inner{y_1}{q_{R_2}}_{L^2(\Omega )} &&\forall\ q_{R_2}\in L^2(\Omega),
 \end{aligned}
\end{equation*}
for given initial values $(y_0,y_1)\in  H^1_0(\Omega) \times L^2(\Omega)$ and data $d\in L^2(q_T)$.

The optimality system then reads as follows:

Find $(y,u,p_U,p_{R_1},p_{R_2})\in Y\times U\times U\times R_1 \times R_2$ such that
\begin{equation}\label{eq:optsystemwave}
 \begin{pmatrix}
  T'\mathcal{I}_O T & 0 & K_U' & K_{R_1}' & K_{R_2}' \\
  0 & \alpha \mathcal{I}_U & \mathcal{I}_U & 0 & 0\\
  K_U & \mathcal{I}_U & 0 & 0 & 0\\
  K_{R_1} & 0 & 0 & 0 & 0\\
  K_{R_2} & 0 & 0 & 0 & 0
 \end{pmatrix}
 \begin{pmatrix}
  y\\u\\p_U\\p_{R_1}\\p_{R_2}
 \end{pmatrix}
=
\begin{pmatrix}
 \inner{d}{T\cdot}_{L^2(q_T)} \\ 0 \\ 0 \\ \inner{y_0}{\cdot}_{H^1_0(\Omega)} \\ \inner{y_1}{\cdot}_{L^2(\Omega)}
\end{pmatrix},
\end{equation}
with the spaces $Y$ given by \cref{eq:defYwave},
\[
 U = L^2(Q_T),\quad R_1 = H^1_0(\Omega),\quad R_2 = L^2(\Omega), \quad O = L^2(q_T),
\]
and the operators
\begin{equation}\label{eq:AbsToWaveMap}
 \begin{aligned}
  T &: Y\longrightarrow O,\quad&&y\mapsto y|_{q_T},\\
  K_U&:Y\longrightarrow U',&&y\mapsto \inner{\partial_{tt}y-\Delta y}{\cdot}_{L^2(Q_T)},\\
  K_{R_1}&:Y\longrightarrow R_1',&& y\mapsto \inner{y(0)}{\cdot}_{H^1_0(\Omega)},\\
  K_{R_2}&:Y\longrightarrow R_2',&& y\mapsto \inner{\partial_t y(0)}{\cdot}_{L^2(\Omega)}.
 \end{aligned}
\end{equation}

It follows from \cref{co:wave} that the system \cref{eq:optsystemwave} is well-posed. Additionally, by \cref{theo:mainAppPrec}, the $\mathcal{P}$-norm leading to an $\alpha$-robust preconditioner is given by
\begin{multline}
  \label{eq:contPrec}
  \norm{( z , v ,q_U,q_{R_1},q_{R_2})}^2_\mathcal{P} \\
 = \norm{ z }^2_{L^2(q_T)} + \alpha \norm{\wave{ z }}^2_{L^2(Q_T)}+\norm{\nabla  z (0)}^2_{\LLO}+\norm{\partial_t  z (0)}^2_{L^2(\Omega)} \\
+ \alpha \norm{ v }^2_{L^2(Q_T)} + \frac{1}{\alpha}\norm{q_U}^2_{L^2(Q_T)} + \norm{\nabla q_{R_1}}^2_{\LLO} + \norm{q_{R_2}}^2_{\LLO}.
\end{multline}

\section{Discretization and numerical experiments} \label{sec:numerics}
In order to illustrate the theoretical results we shortly discuss in this section (as one selected example) the discretization of the optimality system \cref{eq:optsystemwave} of the optimal control problem for the wave equation and present some first numerical results.

\subsection{Discretization}\label{ssec:disc}
We consider conforming discretization spaces; that is,
\begin{equation} \label{eq:conform}
Y_h\subset Y,\quad 
U_h\subset L^2(Q_T), \quad
R_{1,h}\subset H^1_0(\Omega) \quad
\text{and}\quad 
R_{2,h}\subset L^2(\Omega).
\end{equation}
Let $K_{U_h}$, $K_{R_{1,h}}$, $K_{R_{2,h}}$ be the matrix representations of the linear operators $K_{U}$, $K_{R_1}$, $K_{R_2}$ defined in \cref{eq:AbsToWaveMap}, on $Y_h$, $U_h$, $R_{h,1}$, $R_{h,2}$ relative to the chosen bases in these spaces, respectively.
Let $M_{Q_T,h}$ and $M_{q_T,h}$ be the matrix representations of the linear operators $\mathcal{I}_U$ and $T' \mathcal{I}_O T$ on $U_h$ and $Y_h$, respectively. 
These matrices are mass matrices, they represent the inner products $\inner{\cdot}{\cdot}_{L^2(Q_T)}$ and $\inner{\cdot}{\cdot}_{L^2(q_T)}$ on $U_h$ and $Y_h$, respectively.

Applying Galerkin's principle to \cref{eq:optsystemwave} leads to the following linear problem:

Find $(y_h, u_h, p_{U_h}, p_{R_{1,h}}, p_{R_{2,h}}) \in Y_h \times U_h \times U_h \times R_{1,h} \times R_{2,h}$ such that
\begin{equation}
  \label{eq:discSysMat}
  \begin{pmatrix}
    M_{q_T,h} & 0 & K^T_{U_h}& K^T_{R_{1,h}}& K^T_{R_{2,h}}\\
    0 & \alpha M_{Q_T,h} & M_{Q_T,h} & 0 & 0\\
    K_{U_h} &  M_{Q_T,h} & 0 & 0 & 0\\
    K_{R_{1,h}} &  0  & 0 & 0 & 0\\
    K_{R_{2,h}} &  0  & 0 & 0 & 0\\
  \end{pmatrix}
    \begin{pmatrix}
    \underline{y}_h\\
    \underline{u}_h\\
    \underline{p}_{U_h}\\
    \underline{p}_{R_{1,h}}\\
    \underline{p}_{R_{2,h}}    
    \end{pmatrix}
    =
    \begin{pmatrix}
    \underline{d}_h\\
    0\\
    0\\
    \underline{y}_{1,h}\\
    \underline{y}_{2,h}    
    \end{pmatrix},
\end{equation}
where $d_h$, $y_{2,h}$ are $L^2$ projections of $d$, $y_2$ on $Y_h$, $R_{2,h}$, respectively, $y_{1,h}$ is the $H^1_0$ projection of $y_1$ on $R_{1,h}$, and underlined quantities denote the vector representations of the corresponding functions from $Y_h$, $U_h$, $R_{1,h}$, $R_{2,h}$ relative to the chosen bases in these spaces. 

Motivated by the analysis of the continuous problem we propose the following block diagonal preconditioner, which is the matrix representation of the $\mathcal{P}$-inner product on the discretization spaces:
\begin{equation}
  \label{eq:precDisc}
  \mathcal{P}_h = 
  \begin{pmatrix}
   P_{Y_h} & 0 & 0 & 0 & 0\\
   0 & \alpha \, P_{U_h} & 0 & 0 & 0\\
   0 & 0 & \alpha^{-1} \,  P_{U_h} & 0 & 0\\
   0 &  0  & 0 & P_{R_{1,h}} & 0\\
   0 &  0  & 0 & 0 & P_{R_{2,h}}\\
 \end{pmatrix}
 \end{equation}
 with
 \begin{align*}
   \dual{P_{Y_h} \underline{y}_h}{\underline{z}_h} &= \inner{y_h}{z_h}_{L^2(q_T)}+ \alpha \inner{\partial_{tt} y_h-\Delta y_h}{\partial_{tt}z_h - \Delta z_h }_{L^2(Q_T)}\\
   & \quad {} + \inner{\nabla y_h(0)}{\nabla z_h(0)}_{L^2(\Omega)} +\inner{\partial_t y_h(0)}{\partial_t  z_h(0)}_{L^2(\Omega)},\\
   \dual{P_{U_h} \underline{u}_h}{\underline{v}_h} &=  \inner{u_h}{v_h}_{\LLQ},\\
   \dual{P_{R_{1,h}} \underline{p}_{R_{1,h}}}{\underline{q}_{R_{1,h}}} &=  \inner{\nabla p_{R_{1,h}}}{\nabla q_{R_{1,h}}}_{\LLO},\\
    \dual{P_{R_{2,h}} \underline{p}_{R_{2,h}}}{\underline{q}_{R_{2,h}}} &= \inner{p_{R_{2,h}}}{q_{R_{2,h}}}_{\LLO}.
 \end{align*}
 
The preconditioner $\mathcal{P}_h$ is a symmetric and positive definite block diagonal matrix. This matrix is also sparse provided basis functions with local support are chosen.

Observe that the $\alpha$-robust preconditioner which results from applying \cref{theo:mainAppPrec} directly to the discrete problem is similar to \cref{eq:precDisc}, but with $P_{Y_h}$ replaced by
\begin{align}
  \begin{split}
       \dual{\widetilde{P}_{Y_h} \underline{y}_h}{\underline{z}_h} &= \inner{y_h}{z_h}_{L^2(q_T)}+ \alpha \dual{K'_{U_h}\mathcal{I}_{U_h}^{-1}K_{U_h} y_h}{z_h}\\
   &+ \langle K'_{R_{1,h}}\mathcal{I}_{R_{1,h}}^{-1}K_{R_{1,h}} y_h, z_h\rangle + \langle K'_{R_{2,h}}\mathcal{I}_{R_{2,h}}^{-1}K_{R_{2,h}} y_h,z_h\rangle.
   \end{split}
\end{align}

In general, $\widetilde{P}_{Y_h}$ is not sparse. Therefore, the application of the corresponding block diagonal preconditioner $\widetilde{\mathcal{P}}_h$ is rather costly. On the other hand, the choice $\widetilde{P}_{Y_h}$ would ensure $\alpha$-robustness of $\widetilde{\mathcal{P}}_h$ by \cref{theo:mainAppPrec} provided \cref{K1} and \cref{K2} hold.  
In the next lemma we present sufficient conditions on the discretization spaces, $Y_h$, $U_h$, $R_h$ which ensure that $P_{Y_h} = \widetilde{P}_{Y_h}$ as well as \cref{K1}.  
\begin{lemma}
   \label{lemma:sparseYh}
   Assume the discretization space $Y_h$ is of the following form
   \begin{equation} \label{eq:tensor}
   Y_h = Y^t_h\otimes Y^x_h,
   \end{equation}
   where $Y^t_h$ is the time discretization and $Y^x_h$ space discretization.
   If the following conditions hold
   \begin{equation} \label{eq:inclusion}
   (\partial_{tt} -\Delta) Y_h \subset U_h, \quad Y^x_h \subset R_{1,h},  \quad Y^x_h \subset R_{2,h},
   \end{equation}
   then 
   $
   \widetilde{P}_{Y_h} = P_{Y_h}
   $
   and \cref{K1} holds for the discretized state equation with the same constant $c_K$ as for the continuous state equation.
 \end{lemma}
 \begin{proof}
   Let $y_h \in Y_h$ be arbitrary but fixed. Using Lemma \ref{le:appendix1} we get
   \begin{align*}
     \dual{K'_{U_h}\mathcal{I}_{U_h}^{-1}K_{U_h} y_h}{y_h} = \sup_{u_h\in U_h}\frac{\dual{K_{U_h}y_h}{u_h}^2}{\inner{u_h}{u_h}_{L^2(Q_T)}} = \sup_{u_h\in U_h}\frac{\inner{(\partial_{tt}-\Delta) y_h}{u_h}^2_{L^2(Q_T)}}{\inner{u_h}{u_h}_{L^2(Q_T)}}.
   \end{align*}
   Since $(\partial_{tt} -\Delta) Y_h \subset U_h$, the supremum is attained for $u_h = (\partial_{tt}-\Delta) y_h$, and we have
   \begin{align*}
     \sup_{u_h\in U_h}\frac{\inner{(\partial_{tt}-\Delta) y_h}{u_h}^2_{L^2(Q_T)}}{\inner{u_h}{u_h}_{L^2(Q_T)}}
       = \|(\partial_{tt}-\Delta) y_h\|_{L^2(Q_T)}^2 .
   \end{align*}
   Therefore, 
   \[
     \dual{K'_{U_h}\mathcal{I}_{U_h}^{-1}K_{U_h} y_h}{y_h}
       = \|(\partial_{tt}-\Delta) y_h\|_{L^2(Q_T)}^2 .
   \]
   Similarly it follows that
   \[
     \langle K'_{R_{1,h}}\mathcal{I}_{R_{1,h}}^{-1}K_{R_{1,h}} y_h, y_h\rangle
      = \|y_h(0)\|_{H_0^1(\Omega)}^2
    \]
    and
    \[
     \langle K'_{R_{2,h}}\mathcal{I}_{R_{2,h}}^{-1}K_{R_{2,h}} y_h, y_h\rangle
      = \|\partial_t y_h(0)\|_{L^2(\Omega)}^2 .
   \]
This shows that the $P_{Y_h}$-norm and the $\widetilde{P}_{Y_h}$-norm coincide and, therefore, the associated inner products coincide. This implies $P_{Y_h} = \widetilde{P}_{Y_h}$.

The three identities from above can be rewritten as
\[
  \|K_{U_h} y_h\|_{U_h'} = \|K_{U} y_h\|_{U'}
\]
and
\[
  \|K_{R_{1,h}} y_h\|_{R_{1,h}'} = \|K_{R_1} y_h\|_{R_1'}, \quad
  \|K_{R_{2,h}} y_h\|_{R_{2,h}'} = \|K_{R_2} y_h\|_{R_2'},
\]
from which it immediately follows that \cref{K1} for the continuous state operator $K$ implies \cref{K1} for the state operator $K_h = (K_{U_h},K_{R_{1,h}},K_{R_{2,h}})^\top$ of the discretized problem.
   \end{proof}

So, under the assumptions of \cref{lemma:sparseYh}, the preconditioner $\mathcal{P}_h = \widetilde{\mathcal{P}}_h$ is sparse and \cref{K1} holds for the discretized problem with a constant independent of the discretization spaces. 
In order to apply \cref{theo:mainAppPrec} and ensure $\alpha$-robustness of the preconditioner \cref{K2} resp.~ \cref{K2'} is required. The particular form of $Y_h$ in \cref{lemma:sparseYh} essentially means that a time-marching scheme for the spatially discretized wave equation is used. Each reasonable scheme of this form produces well-defined approximate solutions for a prescribed right-hand side of the wave equation and prescribed initial data. This ensures the surjectivity of the operators as required in \cref{K2} and \cref{K2'}. Consequently, the preconditioner is $\alpha$-robust, but it is not clear how the constants $c_R$  in \cref{K2} and \cref{K2'} depend on the discretization spaces. An analysis of this dependency is beyond the scope of this paper. Instead we present in the next section numerical experiments for a particular choice of the discretization spaces and report on promising preliminary numerical results.

\subsection{Numerical results}\label{ssec:results}
We consider the optimal control problem from \cref{ssec:wave} with
$\Omega = (0,1)^2$, $\omega = (1/4,3/4)^2$, $T = 1$
and homogeneous data. The following discretization spaces are used:
\begin{align*}
Y_h &= S_{p,\ell}(0,T)\otimes \left[ S_{p,\ell}(\Omega)\cap H^1_0(\Omega) \right],\\
U_h &= S_{p,\ell,p-3}(0,T)\otimes S_{p,\ell,p-3}(\Omega),\\
R_{1,h} &= S_{p,\ell}(\Omega)\cap H^1_0(\Omega),\\
R_{2,h} &= S_{p,\ell}(\Omega).
\end{align*}
Here, $S_{p,\ell, k}(a,b)$ denotes the space of splines of degree $p$ on an equidistant knot span of the interval $(a,b)$ of mesh size $h = (b-a)/2^\ell$ which are $k$-times continuously differentiable. Spline spaces of maximal continuity, i.e., $k = p-1$ are denoted $S_{p,\ell}(a,b)$.
Spline spaces on $\Omega$ are defined as tensor products of univariate splines spaces.
It is easy to see that the chosen discretization spaces satisfy \cref{eq:conform}, \cref{eq:tensor} and \cref{eq:inclusion} for spline degree $p \ge 2$. 

We use the sparse preconditioner $\mathcal{P}_h$ from  \cref{eq:precDisc}. The application of the preconditioner $\mathcal{P}_h$ requires the multiplication of the inverses of its diagonal blocks with vectors. The action of the inverse of $P_{U_h}$ and  $P_{R_{1,h}}$ is efficiently computed by exploiting the tensor product structure and computing the inverse of univariate mass matrices. For $P_{Y_h}$ and $P_{R_{2,h}}$ sparse direct solvers are used.

The preconditioned system is solved using the minimal residual method (MINRES) with random initial starting vector. The stopping criteria is the reduction of the Euclidean norm of initial residual error by a factor of $10^{-8}$.

\cref{t:ItNumEq1} for $p=2$ and \cref{t:ItNumEq2} for $p=3$ show the degree of freedoms (DoFs) of the systems for several levels $\ell$ of refinements and the iteration numbers of MINRES for different values of $\alpha$.

\begin{table}[tbhp]
{\footnotesize
\caption{Iteration numbers using $p = 2$ 
}\label{t:ItNumEq1}
\begin{center}
  \begin{tabular}{| c || c | c | c | c | r |}
    \hline
    $\ell \backslash \alpha$ & $10^0$ & $10^{-3}$ & $10^{-6}$ & $10^{-9}$ & DoFs\\ \hline \hline
    $2$  &   88 & 36& 51& 21&      3\ 604\\  \hline   
    $3$  &  169 & 38& 48& 33&     28\ 452\\  \hline
    $4$  &  255 & 35& 43& 46&    226\ 372\\  \hline
    $5$  &  380 & 39& 36& 53& 1\ 806\ 468\\  \hline
  \end{tabular}
\end{center}
}
\end{table}

\begin{table}[tbhp]
{\footnotesize
\caption{Iteration numbers using $p = 3$
}\label{t:ItNumEq2}
\begin{center}
  \begin{tabular}{| c || c | c | c | c | r |}
    \hline
    $\ell \backslash \alpha$   & $10^0$ & $10^{-3}$ & $10^{-6}$ & $10^{-9}$ & DoFs\\ \hline \hline
    $2$  &  126 & 37& 53& 57&      4\ 643\\  \hline   
    $3$  &  189 & 35& 48& 62&     32\ 343\\  \hline
    $4$  &  302 & 38& 43& 59&    241\ 439\\  \hline
    $5$  &  381 & 38& 39& 55& 1\ 865\ 775\\  \hline
  \end{tabular}
\end{center}
}
\end{table}
 
Reasonably small iteration numbers were observed for $0 < \alpha \ll 1$. For $\alpha = 1$, the iteration numbers are significantly larger. As expected the performance of MINRES does not deteriorate for small values $\alpha$. The dependence on the mesh size $h$ is moderate. 

\begin{remark}
For more complex domains isogeometric analysis, see c.f. \cite{VeiBufSanVaz14,HugCotBaz05}, can be used to obtain smooth conforming discretization subspaces, and multi-patch domains can be dealt with with methods described in 
\cite{BirKap19} and the references within. For large-scale problems sparse direct solvers eventually fail due to memory limitations. The methods described in \cite{GeoHof19} can then be considered.
\end{remark}

\appendix
\section{Auxiliary results}\label{sec:app}
The calculations in \cref{sec:optimalcontrol} rely on three auxiliary lemmas, frequently used in literature, which are listed in the following. For convenience of the reader we present the proof of the first two and refer to \cite{AxeGus83} for the third one.
\begin{lemma}\label{le:appendix1}
 Let $A:V\longrightarrow V'$ and $B:V\longrightarrow Q'$ be linear operators, where $V$ and $Q$ are Hilbert spaces with dual spaces $V'$ and $Q'$.
 Additionally assume that $A$ is self-adjoint and coercive. Then we have
 \[
  \dual{BA^{-1}B' q}{q} = \sup_{0\neq v\in V}\frac{\dual{Bv}{q}^2}{\dual{Av}{v}}\quad\text{for all} \ q\in Q.
 \]
\end{lemma}
\begin{proof}
 Observe that
 \[
  \dual{BA^{-1}B' q}{q} = \dual{B'q}{A^{-1}B'q} = \norm{B'q}^2_{A^{-1}}
 \]
with the norm $\norm{f}_{A^{-1}}= \inner{f}{f}^{1/2}_{A^{-1}}$ on $V'$, given by the inner product $\inner{f}{g}_{A^{-1}}=\dual{f}{A^{-1}g}$ on $V'$. By Cauchy's inequality it easily follows for any inner product and associated norm
\[
 \norm{f} = \sup_{0\neq g}\frac{\inner{f}{g}}{\norm{g}}.
\]
In particular, we have
\[
 \norm{B'q}^2_{A^{-1}} = \sup_{0\neq g\in V'}\frac{\inner{g}{B'q}^2_{A^{-1}}}{\norm{g}^2_{A^{-1}}} = \sup_{0\neq g\in V'}\frac{\dual{g}{A^{-1}B'q}^2}{\dual{g}{A^{-1}g}}.
\]
By substituting $g$ by $Av$ it follows that
\[
 \sup_{0\neq g\in V'}\frac{\dual{g}{A^{-1}B'q}^2}{\dual{g}{A^{-1}g}}
 = \sup_{0\neq v\in V}\frac{\dual{Av}{A^{-1}B'q}^2}{\dual{Av}{v}}
 = \sup_{0\neq v\in V}\frac{\dual{B'q}{v}^2}{\dual{Av}{v}}
 = \sup_{0\neq v\in V}\frac{\dual{Bv}{q}^2}{\dual{Av}{v}}.
\]
\end{proof}
\begin{lemma}\label{le:appendix2}
 Let $A:V\longrightarrow V'$, $B:V\longrightarrow Q'$, and $C:Q\longrightarrow Q'$ be linear operators, where $V$ and $Q$ are Hilbert spaces with dual spaces $V'$ and $Q'$. Additionally assume that $A$ and $C$ are self-adjoint and coercive. Then the condition
 \[
  \dual{BA^{-1}B'q}{q}\leq\dual{Cq}{q}\quad\text{for all} \ q\in Q
 \]
is equivalent to the condition
\[
 \dual{B'C^{-1}Bv}{v}\leq \dual{Av}{v}\quad\text{for all} \ v\in V.
\]
\end{lemma}
\begin{proof}
 Using \cref{le:appendix1} it immediately follows that the first condition is equivalent to the condition
 \[
  \dual{Bv}{q}^2\leq \dual{Av}{v}\dual{Cq}{q} \quad\text{for all} \ v\in V,\, q\in Q
 \]
 and the second condition is equivalent to the condition
 \[
  \dual{B'q}{v}^2 \leq \dual{Av}{v}\dual{Cq}{q}\quad \text{for all} \ v\in V,\, q\in Q.
 \]
These two new conditions are obviously equivalent.
\end{proof}
\begin{lemma}\label{le:appendix3}
 Let $V$ and $Q$ be Hilbert spaces with dual spaces $V'$ and $Q'$.
 
 Let $\mathcal{M}:V\times Q\longrightarrow V'\times Q'$ be a self-adjoint and positive definite linear operator of a $2$-by-$2$ block form
 \[
  \mathcal{M} =
  \begin{pmatrix}
   M_{11} & M_{12} \\
   M_{21} & M_{22}
  \end{pmatrix},
 \]
and $\mathcal{D}:V\times Q\longrightarrow V'\times Q'$ be of $2$-by-$2$ block diagonal form with self-adjoint and positive definite diagonal blocks
\[
 \mathcal{D} = 
  \begin{pmatrix}
   D_{11} & 0 \\
   0 & D_{22}
  \end{pmatrix}.
\]
Then  $\mathcal{M} \sim \mathcal{D}$ if and only if
\[
 M_{11} \sim D_{11},\quad M_{22} \sim D_{22}\quad \text{and} \quad M_{11} \lesssim M_{11} - M_{12}M_{22}^{-1}M_{21}.
\]
\end{lemma}

\section*{Acknowledgments}
The work of the first author was supported by the Austrian Science Fund (FWF), project I3661-N27.
The work of the second author was supported by the Austrian Science Fund (FWF), project P31048.
The work of the third author was supported by the Austrian Science Fund (FWF), project S11702-N23.

\bibliographystyle{siamplain}
\bibliography{bibliography.bib}
\end{document}